\documentclass{amsart}
\usepackage{amssymb}
\usepackage{amsfonts}
\usepackage{amssymb}
\usepackage{amsmath}
\usepackage{amsthm}
\usepackage{enumerate}
\usepackage{tabularx}
\usepackage{centernot}
\usepackage{mathtools}
\usepackage{stmaryrd}
\usepackage{amsthm,amssymb}
\usepackage{etoolbox}
\usepackage{tikz}
\usepackage{amssymb}
\usetikzlibrary{matrix}
\usepackage{tikz-cd}
\usepackage{tikz}
\usepackage{marginnote}
\definecolor{mygray}{gray}{0.85}
\usepackage[backgroundcolor=mygray,colorinlistoftodos,prependcaption,textsize=small]{todonotes}
\usepackage{xargs}

\newcommand{\dist}{\mathrm{dist}}
\newcommand{\Cay}{\mathrm{Cay}}
\newcounter{claimcounter}
\newenvironment{claim}{\stepcounter{claimcounter}{\noindent {\underline{\em Claim \theclaimcounter}.}}}{}
\newenvironment{claimproof}[1]{\noindent{{\em Proof.}}\space#1}{\hfill $\rule{0.40em}{0.40em}$}

\renewcommand{\leq}{\leqslant}
\renewcommand{\geq}{\geqslant}

\renewcommand{\trianglelefteq}{\trianglelefteqslant}

\newtheorem{theorem}{Theorem}[section]

\newtheorem{corollary}[theorem]{Corollary}
\newtheorem{definition}[theorem]{Definition}
\newtheorem{lemma}[theorem]{Lemma}

\newtheorem{proposition}[theorem]{Proposition}
\newtheorem{example}[theorem]{Example}

\newtheorem{fact}[theorem]{Fact}
\newtheorem{remark}[theorem]{Remark}
\newtheorem{notation}[theorem]{Notation}

\newtheorem*{theorem1.2}{Theorem~1.2}
\newtheorem*{problem1.5}{Problem~1.5}
\newtheorem*{conjecture1.4}{Conjecture~1.4}
\newtheorem*{theorem1.4}{Theorem~1.5}
\newtheorem*{theorem1.11}{Theorem~1.11}
\newtheorem*{conjecture1.6}{Conjecture~1.6}
\newtheorem*{theorem1.7}{Theorem~1.7}
\newtheorem*{conjecture1.7}{Conjecture~1.7}
\newtheorem*{conjecture1.8}{Conjecture~1.8}
\newtheorem*{conjecture1.9}{Conjecture~1.9}
\newtheorem*{conjecture1.10}{Conjecture~1.10}
\newtheorem*{corollary1.3}{Corollary~1.3}

\usepackage{eso-pic}
\begin{document}

\begin{abstract} We lay the foundations of the first-order model theory of Coxeter groups. Firstly, with the exception of the $2$-spherical non-affine case (which we leave open), we characterize the superstable Coxeter groups of finite rank, which we show to be essentially the Coxeter groups of affine type.
Secondly, we characterize the Coxeter groups of finite rank which are domains, a central assumption in the theory of algebraic geometry over groups, which in many respects (e.g. $\lambda$-stability) reduces the model theory of a given Coxeter system to the model theory of its associated irreducible components.
In the second part of the paper we move to specific definability questions in right-angled Coxeter groups (RACGs) and $2$-spherical Coxeter groups. In this respect, firstly, we prove that RACGs of finite rank do not have proper elementary subgroups which are Coxeter groups, and prove further that reflection independent ones do not have proper elementary subgroups at all. Secondly, we prove that if the monoid $Sim(W, S)$ of $S$-self-similarities of $W$ is finitely generated, then $W$ is a prime model of its theory. 
Thirdly, we prove that in reflection independent RACGs  of finite rank the Coxeter elements are type-determined.
We then move to $2$-spherical Coxeter groups, proving that if $(W, S)$ is irreducible, $2$-spherical even and not affine, then $W$ is a prime model of its theory, and that if $W_{\Gamma}$ and $W_{\Theta}$ are as in the previous sentence, then $W_{\Gamma}$ is elementary equivalent to $W_{\Theta}$ if and only if $\Gamma \cong \Theta$, thus solving the elementary equivalence problem for most of the $2$-spherical Coxeter groups.
In the last part of the paper we focus on model theoretic applications of the notion of reflection length from Coxeter group theory, proving in particular that \mbox{affine Coxeter groups are not connected.}
\end{abstract}

\title{First-Order Aspects of Coxeter Groups}
\thanks{The second and third author were partially supported by European Research Council grant 338821. The second author was partially supported by project PRIN 2017 ``Mathematical Logic: models, sets, computability", prot. 2017NWTM8R. The third author was partially supported by Israel Science Foundation (ISF) grant no: 1838/19. No. 1201 on Shelah's publication list.}

\author{Bernhard M\"uhlherr}
\address{Mathematisches Institut Justus-Liebig-Universit{\"a}at Gie{\ss}en, Gie{\ss}en, Germany}

\author{Gianluca Paolini}
\address{Department of Mathematics ``Giuseppe Peano'', University of Torino, Italy.}

\author{Saharon Shelah}
\address{Einstein Institute of Mathematics,  The Hebrew University of Jerusalem, Israel \and Department of Mathematics,  Rutgers University, U.S.A.}

\date{\today}
\maketitle

\tableofcontents

\section{Introduction}

	Since the work of Sela \cite{sela_elem_equiv}, and Kharlampovich \& Myasnikov \cite{myasnikov_elem_equiv} on Tarski's problem for non-abelian free groups, the model theoretic analysis of classes of groups arising from combinatorial and geometric group theory has seen crucial advancements, famous is for example the extension of the methods employed for free groups to the analysis of the model theory of torsion-free \mbox{hyperbolic groups \cite{myasnikov_hyper, sela_hyper}.}
	
	In the present study we lay the foundations of the first-order model theory of {\em Coxeter groups}, a class of groups that arises in a multitude of ways in several areas of mathematics, such as algebra \cite{humphreys}, geometry \cite{davis} and combinatorics \cite{brenti}.  
	This area of model theory is a largely unexplored territory. In fact, at the best of our knowledge, the only known results on the first-order\footnote{See \cite{paolini_coxeter} for a model theoretic analysis of right-angled Coxeter groups in the non first-order context of abstract elementary classes.} model \mbox{theory of Coxeter groups are:}
	\begin{fact}
	\begin{enumerate}[(a)]
	\item If $W_{\Gamma}$ and $W_{\Theta}$ are two right-angled Coxeter groups of finite rank, then $W_{\Gamma}$ is elementary equivalent to $W_{\Theta}$ if and only if $\Gamma \cong \Theta$ (see \cite{casal1}).
	\item If $W$ is a right-angled Coxeter group of finite rank, then the existential (resp. positive) first-order theory of $W$ is decidable (see \cite{decidability_positive}).
\end{enumerate}
\end{fact}

	Our first main result is a nearly complete characterization of the {\em superstable} Coxeter groups of finite rank (we leave open the non-affine $2$-spherical case). The property of superstability is one of the main diving lines in model theory, with very strong structural consequences, e.g. its negation implies the maximal number of models in every uncountable cardinality. Most abelian groups are superstable \cite{rogers} (e.g. the free ones), while e.g. non-abelian free groups are not superstable \cite{poizat}\footnote{We observe that the unsuperstability of non-abelian free groups has been known since at least the 80's (see in fact \cite{poizat} for an explicit proof of this fact). On the other hand, the stability of non-abelian free groups has been established only recently by the fundamental work of Sela \cite{sela_stability}.}.

	\begin{theorem}\label{th_char_sstab} Let $(W, S)$ be a Coxeter system of finite rank, and let $W_1 \times \cdots \times W_n$ be the corresponding decomposition of $W$ into irreducible components. Suppose further that $W$ is infinite and not non-affine $2$-spherical. Then $W$ is superstable if and only if, for every $i \in [1, n]$, $W_i$ is \mbox{either of affine type or of spherical type.}
\end{theorem}

	The negative side of Theorem~\ref{th_char_sstab} actually follows from an abstract technical criterion, Theorem~\ref{general_criterion}, which is of independent interest and whose applicability is well-beyond the present case study. In particular, it also allows us to prove:
	
	\begin{theorem}\label{pure_theorem} Let $G$ be a group. A sufficient condition for the unsuperstability of $G$ is that, for {\em some} $2 \leq n < \omega$, there exists a non-abelian free subgroup $\mathbb{F} \leq G$ which is $n$-pure in $G$, i.e. for every $x \in G$, if $x^n \in \mathbb{F}$, then $x \in \mathbb{F}$. In particular, if $G$ is a virtually non-abelian free group, then $G$ is not superstable.
\end{theorem}

	The generality of Theorem~\ref{pure_theorem} also implies a characterization of the superstable right-angled Artin groups of finite rank (on these groups see also below):
	
	\begin{corollary}\label{Artin_theorem} Let $A_\Gamma = A$ be an Artin group, if $\Gamma$ contains a non-edge (i.e. the associated Coxeter matrix has an $\infty$ entry), then $A$ is not superstable. In particular, a right-angled Artin group is \mbox{superstable if and only if it is abelian.}
\end{corollary}

	Also the positive side of Theorem~\ref{th_char_sstab} follows from a stronger result:
	
	\begin{theorem}\label{decidability_affine_Coxeter} Let $(W, S)$ be an irreducible affine Coxeter group. Then $W$ is intepretable in $(\mathbb{Z}, +, 0)$ with finitely many parameters, and so $Th(W)$ is decidable.
\end{theorem}

	Our second main result concerns {\em algebraic geometry over groups}, a general theory developed in a series of papers \cite{alg_geom_over_groups1, alg_geom_over_groups2, alg_geom_over_groups3} which has important algebraic and model theoretic applications. In order to develop this involved machinery, the authors of \cite{alg_geom_over_groups1, alg_geom_over_groups2, alg_geom_over_groups3} isolate a group theoretic property, that of being a {\em domain} (cf. Definition~\ref{def_domain}), and show that under this assumption many notions from algebraic geometry can be developed in a purely group theoretic context. We prove:
	


	\begin{theorem}\label{domain_conj} Let $(W, S)$ be a Coxeter system of finite rank, and $W_1 \times \cdots \times W_n$ the corresponding decomposition of $W$ into irreducible components. Suppose that for every $i \in [1, n]$, $W_i$ is neither spherical nor affine, then the following are equivalent:
	\begin{enumerate}[(1)]
	\item $W$ is a domain;
	\item $n=1$, i.e. $(W, S)$ is irreducible.
	\end{enumerate}
\end{theorem}

	By a case-by-case analysis of the irreducible affine and spherical Coxeter groups  it could be seen that $W$ is a domain if and only if $W$ is irreducible, centerless and not of affine type, but this is outside of the scope of the present paper.
	Using Theorem~\ref{domain_conj} in combination with the general results of \cite{alg_geom_over_groups3}, we then deduce:

	\begin{corollary}\label{cor1} Let $(W, S)$ be a Coxeter system of finite rank, and suppose that $(W, S)$ is irreducible and neither spherical nor affine. Then the machinery of algebraic geometry over groups \cite{alg_geom_over_groups1, alg_geom_over_groups2, alg_geom_over_groups3} can be applied to $W$.
\end{corollary}

	\begin{corollary}\label{reduction_to_components} Let $(W, S)$ be a Coxeter system of finite rank, and let $ W_1 \times \cdots \times W_n$ be the corresponding decomposition of $W$ into irreducible components, and suppose that, for every $i \in [1, n]$, $W_i$ is neither spherical nor affine. Then:
	\begin{enumerate}[(i)]
	\item if $W \equiv H$, then $H$ is also a finite direct product of domains $H = H_1 \times \cdots \times H_k$, with $k = n$ and $W_i \equiv H_i$, for all $i \in [1, n]$ (after suitable ordering of factors);
	\item for every $i \in [1, n]$, $Th(W_i)$ is interpretable in $W$;
	\item $W$ is $\lambda$-stable if and only if $W_i$ is $\lambda$-stable for every $i \in [1, n]$;
	\item $Th(W)$ is decidable if and only if $Th(W_i)$ is decidable for every $i \in [1, n]$.
\end{enumerate}
\end{corollary}

	In the second part of the paper we move to definability questions in specific classes of Coxeter groups. The first class that we consider is the class of {\em right-angled Coxeter groups} (RACGs). These groups play a central role in the theory of Coxeter groups (and related objects), and they are closely related  to the so-called right-angled Artin group (RAAGs), a class of groups whose model theoretic analysis has recently seen important advancements (see e.g. \cite{artin1, artin2}).
	
	We focus on two model theoretic problems: the analysis of elementary substructures of a given structure, and the question of primality of a given model (where a model $M$ is prime if it embeds elementarily in every model of its theory).
	
	The question of elementarity of a subgroup of a given group is a classical theme in model theory. For example, in the 50's Tarski asked if the natural embedding of the non-abelian free group $\mathbb{F}_k$ on $k$ generators into the non-abelian free group $\mathbb{F}_n$ on $n$ generator was elementary ($k \leq n < \omega$). This was settled in the positive by Sela \cite{sela_elem_equiv}, and, independently, by Kharlampovich \& Myasnikov \cite{myasnikov_elem_equiv}. Furthermore, Perin proved that $H$ is elementary in $\mathbb{F}_n$ if and only if $H$ is a free factor of $\mathbb{F}_n$ \cite{perin_el1, perin_el2}.
	
	In the following theorem we give a strong contribution to the question of determination of the elementary substructures of a right-angled Coxeter group $W$, and a full answer in the case $W$ is further assumed to be reflection independent, i.e. its set of reflections $S^W$ is independent of the choice of Coxeter basis $S$ of $W$.

	\begin{theorem}\label{el_substr} Let $W$ be a right-angled Coxeter group of finite rank. Then $W$ does not have proper elementary subgroups which are Coxeter groups. Furthermore, if $W$ is reflection independent, then the set of reflections of $W$ is definable without parameters and $W$ has {\em no} proper elementary subgroups at all.
\end{theorem}

	Another classical theme in model theory is the existence of prime models, and the question of homogeneity of a given model, where a structure is said to be homogeneous if tuples realizing the same first-order types over $\emptyset$ are automorphic\footnote{The two questions are stricly related, since a prime model $M$ is such that orbits of tuples are not only type-definable over $\emptyset$, but actually first-order definable over $\emptyset$, and thus $M$ is homogenous.}. In \cite{nies} Nies proved that $\mathbb{F}_2$ is homogeneous and that the theory of non-abelian free groups does not have a prime model. Nies left open the question of homogeneity of free groups of finite rank $\geq 3$, which was solved in the positive only about 15 years later by Sklinos \& Perin \cite{homogeneity}, and, independently, by Houcine \cite{houcine}.
	
	In Theorem~\ref{th_prime_models} we connect the question of primality of a given $\mathrm{RACG}$ of finite rank $W$ to the finite generation of a certain monoid of monomorphisms of $W$: the monoid of $S$-self-similarity of $W$, denoted as $Sim(W, S)$ (cf. Definition~\ref{special_monoid}).
	
	\begin{theorem}\label{th_prime_models}
Let $W$ be a $\mathrm{RACG}$ of finite rank and $S$ a basis of $W$. Then:
	\begin{enumerate}[(1)]
	\item If the monoid $Sim(W, S)$ of special $S$-endomorphisms of $W$ is finitely generated, then $W$ is a prime model of its theory, that is, for every $n < \omega$ and any $n$-tuple $\bar{a}$ from $W$, the $\mathrm{Aut}(W)$-orbit of $\bar{a}$ is definable in $W$ without parameters.
	\item If $Sim(W, S)$ is not finitely generated, then the orbit of (any enumeration of) $S$ is not $\emptyset$-definable in $W$ without parameters by a universal formula.
	\item If $W$ is a universal Coxeter group of finite rank at least two, then the monoid $Sim(W, S)$ is {\em not} finitely generated (and so the conclusion of (2) applies).
\end{enumerate}	 
\end{theorem}

%

	Finally, we prove that right-angled Coxeter groups of finite rank manifest strong traces of homogeneity, leaving though open the problem of full homogeneity.

	\begin{theorem}\label{type_def_th}
Let $W$ be a right-angled Coxeter group of finite rank. Then if $\bar{a} \in W^n$ is such that $\langle \bar{a} \rangle_W$ contains a Coxeter element of $W$, then $\bar{a}$ is type-determined.
\end{theorem}

	
	We now move to another important and well-known class of Coxeter groups which shows a very different model theoretic behaviour: the {\em $2$-spherical Coxeter groups}. In this respect, relying on the fundamental results of \cite{reflec_abc_cox, intrinsic_reflections}, \mbox{we were able to prove:}
	
	\begin{theorem}\label{th_prime_2spherical} Let $(W, S)$ be an irreducible, $2$-spherical Coxeter system of finite rank. Then the set of reflections of $W$ is definable without parameters. Furthermore, if $(W, S)$ is even and not affine, then $W$ {\em is} a prime model of its theory.
\end{theorem}

	\begin{corollary}\label{el_equivalence_corollary} Let $W_{\Gamma}$ and $W_{\Theta}$ be irreducible, $2$-spherical, even and not affine Coxeter groups, then $W_{\Gamma}$ is elementary equivalent to $W_{\Theta}$ if and only if $\Gamma \cong \Theta$.
\end{corollary}

	In the last part of our paper we focus on model theoretic applications of the notions of reflection length, i.e. the study of $W$ with respect to the generating set $S^W$ (the set of $S$-reflections of $W$). Recently, the notion of reflection length has received the attention of several researchers in Coxeter group theory\footnote{Sometimes this area of research goes under the name of ``dual Coxeter theory''.}, see e.g. \cite{bessis, brewster, dyer_reflection, duszenko, kyle}. One of the most important results concerning this notion is that a Coxeter group of finite rank has bounded reflection length iff it is either spherical or affine \cite{duszenko, kyle}, and further that in these cases explicit bounds can be given \cite{brewster}.
	
	We observe that the unboundedness phenomenon just mentioned implies that $\aleph_0$-saturated elementary extensions of infinite non-affine Coxeter groups have ``non-standard elements'', i.e. elements with ``infinite reflection length''. On the other hand, elementary extensions of affine Coxeter groups are always generated by reflections and thus these groups behave very differently. We believe that the boundedness of reflection length in affine Coxeter groups is related to the phenomenon of superstability proved in Theorem~\ref{th_char_sstab} \mbox{(but we have no hard evidence at the moment).}

	\begin{theorem}\label{first_theorem} Let $(W, S)$ be an infinite Coxeter system of finite rank, and let $G$ be an elementary extension of $W$. Let $N_G = N = \langle g \in G: g^2 = e \rangle_G$. Then:
	\begin{enumerate}[(1)]
	\item \label{item1} $N$ is generated by $S^G$, and $N$ is a characteristic subgroup of $G$;
	\item\label{item3} if $W$ is not affine and $G$ is $\aleph_0$-saturated, then $N \neq G$;
	\item\label{item4} if $W$ is affine, then $G = N$.
\end{enumerate}
\end{theorem}

	We then focus on definable subgroups of Coxeter groups of finite rank, showing that kernels of reflection invariant homorphisms of $W$ determine $\bigvee$-definable subgroups of $W$, and that in the case of affine Coxeter groups they actually determine first-order definable subgroups. In particular, we were able to show:

	\begin{corollary}\label{affine_corollary}  Let $(W, S)$ be an affine Coxeter system (of finite rank). Then the alternating subgroup of $(W, S)$ is definable in $W$ over $S$. In particular, the monster model of $W$ has a definable subgroup of index two and so it is {\em not} a connected group.
\end{corollary}

	We conclude the paper combining our methods with the construction from \cite{davis} establishing that right-Angled Artin groups are commensurable with right-angled Coxeter groups, showing that this construction is actually $\bigvee$-definabile. This gives a partial answer to the following question which we consider to be of independent interest: is there a way to establish a technical relation between the model theory of right-angled Artin groups and the model theory of right-angled Coxeter groups?

	
	\begin{corollary}\label{artin_theorem} For any right-angled Artin group $A$ of finite rank there exists a right-angled Coxeter system of finite rank $(W_A, S)$ such that $A$ is a normal subgroup of $W_A$, $A$ has finite index in $W_A$, and $A$ is a $\bigvee$-definable subgroup of $W_A$ over $S$.
\end{corollary}

\section{Model Theoretic Preliminaries}

	For a nice introduction to model theory and a detailed background on the definitions which we are about to introduce see e.g. \cite{marker}. For a text specifically devoted to the model theory of groups see e.g. the classical reference \cite{poizat_stables}. 
	
	The first-order language of group theory consists of the formulas:
\begin{enumerate}[(i)]
	\item atomic expressions of the form $\sigma(\bar{x}) = \tau(\bar{y})$, where $\sigma(\bar{x})$ and $\tau(\bar{y})$ are group theoretic terms (words) in the variable $\bar{x}$ and $\bar{y}$, respectively;
	\item the closure of the atomic formulas from (i) under $\wedge$ (``and''),  $\vee$ (``or''), $\neg$ (``not''), $\forall$ (``for all''), and $\exists$ (``there exists'').
\end{enumerate}

	The occurrence of a variable $x$ in the formula $\varphi$ is said to be free if it is not contained in a subformula of $\varphi$ which is immediately preceded by a quantifier which bounds $x$  (i.e. the symbols $\forall x$ or $\exists x$). We usually denote a first-order formula by $\varphi(\bar{x})$, $\bar{x} = (x_1, . . . , x_n)$, if the free variables which occur in $\varphi$ are among $x_1, . . . , x_n$.
	
	A formula $\varphi$ with no free variables (i.e. a formula in which every occurrence of every variable is not free) is said to be a sentence.
If $\varphi(\bar{x})$, $\bar{x} = (x_1, . . . , x_n)$, is a first-order formula and $\bar{g} \in G^n$, for $G$ a group, we say that $\varphi(\bar{g} )$ is satisfied in $G$, if the formulas $\varphi(\bar{x})$ is true in $G$ under the assignment $x_i \mapsto g_i$, for $i \in [1, n]$. This is denoted by $G \models \varphi(\bar{g})$ (where the symbol $\models$ stands for ``models'').

	\begin{definition} Let $G$ and $H$ be groups.
	\begin{enumerate}[(1)]
	\item We say that $G$ is elementary equivalent to $H$, if a sentence $\varphi$ is true in $G$ if and only if it is true in $H$, i.e. $G$ and $H$ have the same first-order theory.
	\item We say that $G$ is an elementary subgroup of $H$ if $G$ is a subgroup of $H$ and for every formula $\varphi(\bar{x})$, $\bar{x} = (x_1, . . . , x_n)$, and $\bar{g} \in G^n$ we have that:
	$$G \models \varphi(\bar{g}) \; \Leftrightarrow \; H \models \varphi(\bar{g}).$$
In this case we also say that $H$ is an elementary extension of $G$.
	\item We say that $G$ is elementary embeddable into $H$ if there exists an embedding $\alpha: G \rightarrow H$ such that $\alpha(G)$ is an elementary subgroup of $H$.
\end{enumerate}
\end{definition}

	\begin{definition} Let $G$ be a group and $X \subseteq G^n$. Given $P \subseteq G$, we say that $X$ is $P$-definable, or definable over $P$, if there exists a first-order formula $\varphi(\bar{x}, \bar{a})$, with $\bar{a} \in P^m$, such that $X = \{ \bar{g} \in G^n : G \models \varphi(\bar{g}, \bar{a}) \}$. We say that $X$ is definable if it is definable over {\em some} set of parameters (although we sometimes say ``definable with parameters''). We say that $X$ is definable without parameters if it is $\emptyset$-definable.
\end{definition}

	\begin{definition} We say that a group $G$ is a prime model of its theory if it is elementary embeddable in every group $H$ elementary equivalent to it.
\end{definition}

	\begin{definition} Let $G$ be a group and $H$ a subgroup of $G$. We say that $H$ is $\bigvee$-definable (resp. definable, or first-order definable) in $G$ if the following hold:
	\begin{enumerate}[(i)]
	\item $H$ is definable in $G$ by a countable disjunction, i.e. one of size $\leq \aleph_0$,  (resp. by a first-order formula $\varphi(x)$) with parameters from $G$;
	\item in every elementary extension $G'$ of $G$ this disjunction (resp. the first-order formula $\varphi(x)$) defines a subgroup of $G$.
	\end{enumerate}
\end{definition}

	\begin{definition}\label{connected} We say that a group $G$ is connected if it does not have a proper (first-order) definable subgroup of finite index.
\end{definition}

\begin{definition}\label{classical_interpretability} Let $M$ and $N$ be models. We say that $N$ is interpretable in $M$ over $A\subseteq M$ if for some $n < \omega$ there are:
	\begin{enumerate}[(1)]
	\item an $A$-definable subset $D$ of $M^n$;
	\item an $A$-definable equivalence relation $E$ on $D$;
	\item a bijection $\alpha: N \rightarrow D/E$ such that for every $m < \omega$ and $\emptyset$-definable subset $R$ of $N^m$ the subset of $M^{nm}$ given by:
	$$\{ (\bar{a}_1, ..., \bar{a}_m) \in (M^n)^m : (\alpha^{-1}(\bar{a}_1/E), ...,  \alpha^{-1}(\bar{a}_m/E)) \in R\}$$
is $A$-definable in $M$.
\end{enumerate}
\end{definition}

	There are many equivalent definitions of superstability (the notion occurring in Theorem~\ref{th_char_sstab}), we will give one in Section~\ref{sec_superst}, see Definition~\ref{def_unsuper}.

\section{Coxeter Groups}\label{preliminaries_sec}

\begin{definition}[Coxeter groups]\label{def_Coxeter_groups} Let $S$ be a set. A matrix $m: S \times S \rightarrow \{1, 2, . . . , \infty \}$ is called a {\em Coxeter matrix} if it satisfies:
	\begin{enumerate}[(1)]
	\item $m(s, s') = m(s' , s)$;
	\item $m(s, s') = 1 \Leftrightarrow s = s'$.
	\end{enumerate}
For such a matrix, let $S^2_{*} = \{(s, s') \in S^2 : m(s, s' ) < \infty \}$. A Coxeter matrix $m$ determines
a group $W$ with presentation:
$$
\begin{cases} \text{Generators:} \; \;  S \\
				\text{Relations:} \; \;   (ss')^{m(s,s')} = e, \text{ for all } (s, s' ) \in S^2_{*}.
\end{cases} $$
A group with a presentations as above is called a Coxeter group, and the pair $(W, S)$ is a called a Coxeter system. The rank of the Coxeter system $(W, S)$ is $|S|$.
\end{definition}

	\begin{definition}\label{def_Coxeter_graph} In the context of Definition~\ref{def_Coxeter_groups}, the Coxeter matrix $m$ can equivalently be represented by a labeled graph $\Gamma$ whose node set is $S$ and whose set of edges $E_\Gamma$ is the set of pairs $\{s, s' \}$ such that $m(s, s') < \infty$, with label $m(s, s')$. Notice that some authors consider instead the graph $\Delta$ such that $s$ and $s'$ are adjacent iff $m(s, s ) \geq 3$. In order to try to avoid confusion we refer to the first graph as the Coxeter graph of $(W, S)$ (and usually denote it with the letter $\Gamma$), and to the second graph as the Coxeter diagram of $(W, S)$ (and usually denote it with the letter $\Delta$).
\end{definition}

	\begin{definition}[Right-angled Coxeter and Artin groups] Let $m$ be a Coxeter matrix and let $W$ be the corresponding Coxeter group. We say that $W$ is right-angled if the matrix $m$ has values in the set $\{ 1, 2, \infty\}$. In this case the Coxeter graph $\Gamma$ associated to $m$ is simply thought as a graph (instead of a labeled graph), with edges corresponding to the pairs $\{ s, s' \}$ such that $m(s, s') = 2$. A right-angled Artin group is defined as in the case of right-angled Coxeter groups with the omission in the defining presentation of the requirement that generators have order~$2$.
\end{definition}

\begin{definition}\label{def_basis} Let $W$ be a Coxeter group. We say that $T \subseteq W$ is a Coxeter basis of $W$ (or a Coxeter generating set for $W$), if $(W, T)$ is a Coxeter system for $W$.
\end{definition}

\begin{fact} \label{basicsonJ} Let $(W, S)$ be a Coxeter system and $J \subseteq S$.
\begin{enumerate}[(a)]
\item $(\langle J \rangle_W, J)$ is a Coxeter system;
\item for $K \subseteq S$ we have $\langle J \rangle_W \cap \langle K \rangle_W = \langle J \cap K \rangle_W$;
\item if $t \in S - J$ normalizes $\langle J \rangle_W$, then $[t,J] = e$.
\end{enumerate}
\end{fact}

\begin{proof} Assertions (a) and (b) are well known (see e.g. \cite{humphreys}). Let $t$ be as in Item (c) and $s \in J$. Then $tst \in \langle s,t \rangle_W \cap \langle J \rangle_W = \langle s \rangle_W$, by Item (b), and so $st = ts$.
\end{proof}

	\begin{definition}\label{def_parabolic} Let $(W, S)$ be a Coxeter system. An {\sl $S$-parabolic subgroup} of $W$ is a subgroup $P$ of $W$ such that $P = w \langle J \rangle_W w^{-1}$
for some $w \in W$ and some $J \subseteq S$. A {\sl special $S$-parabolic subgroup} of $W$ is a subgroup $P$ of $W$ such that $P = \langle J \rangle_W$
for some $J \subseteq S$. A subset $J$
of $S$ is called spherical if $\langle J \rangle_W$ is finite.
\end{definition}

%

	\begin{definition}\label{def_irreducible} Let $(W, S)$ be a Coxeter system with Coxeter diagram $\Delta$ (recall Definition~\ref{def_Coxeter_graph}). We say that $(W, S)$ is irreducible if $\Delta$ is connected.
\end{definition}

	\begin{remark} Let $(W, S)$ be a right-angled Coxeter system with Coxeter graph $\Gamma$ (recall Def.~\ref{def_Coxeter_graph}). Then $(W, S)$ is irreducible iff the complement of $\Gamma$ is connected.
\end{remark}

	\begin{fact} Let $(W, S)$  be a Coxeter system of finite rank. Then $W$ can be written uniquely as a product $W_1 \times \cdots \times W_n$ of irreducible special $S$-parabolic subgroups of $W$ (up to changing the order of the factors $W_i$, $i \in [1, n]$). In fact, if $S_1, ..., S_n$ are the connected components of the Coxeter diagram $\Delta$, then $W_i = \langle S_i \rangle_W$.
\end{fact}

	\begin{definition} Let $W$  be a Coxeter group. We say that $W$ is spherical if it is finite. We say that $W$ is affine if it is infinite and it has a representation as a discrete affine reflection group (see e.g. the classical reference \cite{humphreys} for details). 
\end{definition}

	\begin{definition} Let $(W, S)$  be a Coxeter system with Coxeter matrix $m$. We say that the Coxeter system $(W, S)$ is $2$-spherical (resp. even) if $m$ has only finite entries (resp. if $m$ has only even or infinite entries).
\end{definition}


	\begin{definition} Let $W$  be a Coxeter group of finite rank. For any subset $X$ of $W$ we define its {\sl $S$-parabolic closure}
$Pc_S(X)$ as the intersection of all the $S$-parabolic subgroups of $W$ containing $X$.
\end{definition}

	In the following lemma we collect some basic properties concerning the parabolic closure $Pc_S(X)$.
Given a group $G$ and $X \subseteq G$, we denote by $N_G(X)$ and $C_G(X)$ the normalizer of $X$ in $G$ and the centralizer of $X$ in $G$, respectively.

\begin{lemma}\label{parabclosure}
Let $W$  be a Coxeter group of finite rank. For $X \subseteq W$ we have:
\begin{enumerate}[(a)]
\item $Pc_S(X)$ is an $S$-parabolic subgroup of $W$ containing $\langle X \rangle_W$;
\item if $\langle X \rangle_W$ is finite, then $Pc_S(X)$ is a finite $S$-parabolic subgroup of $W$; in particular,
there is a spherical $J \subseteq S$ and $w \in W$ such that $Pc_S(X) = w \langle J \rangle_W w^{-1}$;
\item $N_W(\langle X \rangle_W) \leq N_W(Pc_S(X))$, and in particular $C_W(X) \leq N_W(Pc_S(X))$.
\end{enumerate}
\end{lemma}

\begin{proof} For Item (a) we refer to the discussion in \cite{krammer} following Proposition 2.1.4.
Item (b) is a consequence of Item (a) and the well-known fact that each finite
subgroup of $W$ is contained in a finite $S$-parabolic subgroup of $W$ (see e.g. \cite[Proposition~2.87]{abramenko}). Item (c) follows from the fact that $W$ normalizes the set of its $S$-parabolic subgroups.
\end{proof}

	\begin{definition}[Reflection length]\label{reflection_length} Given a Coxeter system $(W, S)$ , we denote by $\ell_S$ the length of an element from $W$ with respect to the generating set $S$ (so the minimal length of a word in the alphabet $S$ which spells the element $w \in W$). We denote by $\ell_T$ the length of an element from $W$ with respect to the generating set $T := S^W = \{ wsw^{-1} : s \in S, \; w \in W \}$. The latter length is called {\em reflection length}.
\end{definition}

	\begin{fact}[\cite{duszenko, kyle}]\label{reflection_length_fact} Let $(W, S)$ be a Coxeter group of finite rank and $T = S^W$.
	\begin{enumerate}[(1)]
	\item If $W$ is spherical or affine, then the reflection length $\ell_T$ is bounded.
	\item If $W$ is not as in (1), then the reflection length $\ell_T$ is unbounded.
	\end{enumerate}
\end{fact}

	\begin{fact}[Abelianization]\label{abelianization_fact} Let $(W, S)$ be a Coxeter group of finite rank, and let $m$ be the corresponding Coxeter matrix. Let $\sim$ be the equivalence relation on $S$ defined by taking the transitive closure of the relation $s \sim s'$ if $m(s, s')$ is odd. Then:
	$$\alpha: W \rightarrow (\mathbb{Z}/2\mathbb{Z})^{|S/\sim|}: s \mapsto [s]_\sim,$$
is an homomorphism, and in fact $(\mathbb{Z}/2\mathbb{Z})^{|S/\sim|}$ is the abelianization of $W$.
\end{fact}

\begin{definition}\label{def_refl} Let $(W, S)$ be a Coxeter system of finite rank.
We say that $W$ is {\em reflection independent} if $S^W = \{ wsw^{-1} : s \in S, \; w \in W \}$ is invariant under change of Coxeter basis $S$ of $W$ (i.e. if $S$ and $T$ are two such bases, then $S^W = T^W$).
\end{definition}

\subsection{Reflection Subgroups of Even Coxeter Groups}\label{even_sec}

	\begin{definition}\label{def_reflection_sbg} Let $W$ be a Coxeter group.
	\begin{enumerate}[(1)]
	\item\label{item1} If $(W, S)$ is a Coxeter system and $W' \leq W$, we say that $W'$ is an $S$-reflection subgroup of $W$ if $W' = \langle W' \cap S^W \rangle$.
	\item We say that $W' \leq W$ is a reflection subgroup of $W$ if there is a Coxeter basis $S$ of $W$ such that $W'$ is an $S$-reflection subgroup of $W$.
	\end{enumerate}
\end{definition}

\begin{definition}\label{special_monoid} Let $(W, S)$ be a Coxeter system. We say that $\alpha \in End(W)$ is  an $S$-self-similarity (or a special $S$-endomorphism) if for every $s, t \in S$ we have:
	\begin{enumerate}[(1)]
	\item $\alpha(s) \in s^W$;
	\item $o(\alpha(s) \alpha(t)) = o(st)$.
\end{enumerate}
We denote the monoid of $S$-self-similarities as $Sim(W, S)$. We denote the semigroup of $S$-self-similarities which are not automorphisms as $Sim^*(W, S)$. We say that $ U \leq W$ is an $S$-self-similar subgroup if $U = \langle \alpha(S)\rangle_W$, for some $\alpha \in Sim(W, S)$.
\end{definition}

	\begin{example} We give an example of an $\alpha \in Sim(W, S)$. Let $S = \{s_1, ..., s_n\}$, $n \geq 2$ and $W$ be the free Coxeter group with basis $S$. We define $\alpha \in Sim(W, S)$ by letting $\alpha(s_1) = s_2s_1s_2$, $\alpha(s_2) = s_1s_2s_1$ and, for $2 \leq i \leq n$, $\alpha(s_i) = s_i$.
\end{example}

	\begin{definition}\label{def_self_similar} Let $(W,S)$ be a Coxeter system. We say that $\hat{S}$ is a {\em self-similar} set of reflections of $(W,S)$ (or a self-similar set of $S$-reflections) if $\hat{S} = \{ \hat{s} : s \in S \}$ and for all $s,t \in S$ we have $\hat{s} \in s^W$ and $o(\hat{s} \hat{t}) = o(st)$. I.e., $\hat{S}$ is a {\em self-similar} set of reflections of $(W,S)$ if there is $\alpha \in Sim(W, S)$ such that $\alpha(s) = \hat{s}$, for every $s \in S$.
\end{definition}


	We now collect some facts which will be used in the proof of Proposition~\ref{muller_fact}.

	\begin{fact}\label{fact_for_even} Let $(W,S)$ be a Coxeter system of finite rank.
	\begin{enumerate}[(A)]
	\item As the geometric representation of $(W,S)$ is faithful (see e.g. \cite{humphreys}), $W$ is a finitely  generated linear group over the real numbers. It follows that $W$ is a residually finite group, and in particular that it is Hopfian.
	\item If $u, v \in S^W$ are such that $D = \langle u,v \rangle_W$ is finite, then there exist $s,t \in S$ and $w \in W$ such that $o(st)$ is finite and  $D^w \leq \langle s, t \rangle_W$.	This follows from Assertion (d) of Theorem 1.12 in \cite{humphreys} in the case where $W$ is
finite. Using Assertion (b) of Lemma \ref{parabclosure} the general case can be reduced to the spherical case.
	\item If $T \subseteq S^W$ and $U := \langle T \rangle_W$ then there exists $R \subseteq U \cap S^W$ such that $(U,R)$ is a Coxeter system and $R^U = U \cap S^W$ (Theorem (3.3) in \cite{dyer}). Moreover, if $T$ is finite, then $|R| \leq |T|$ (Assertion (i) of Corollary (3.11) in 
\cite{dyer}).
	\item If $(W,S)$ is even, then for $s \neq t \in S$ we have that $s^W \cap t^W = \emptyset$ (follows from Fact \ref{abelianization_fact}).
	\item Suppose $(W,S)$ is even and let $s, t \in S$ be such that $st$ has infinite order. Then $xy$ has infinite order for all $x \in s^W$ and $y \in t^W$. This follows from Facts (B) and (D) above.
\end{enumerate}
\end{fact}

	\begin{proposition}\label{muller_fact} Let $(W,S)$ be an even Coxeter system of finite rank, and let $\alpha$ be a self-similarity of $(W,S)$ (cf. Definition~\ref{special_monoid}). Then $(\langle \alpha(S) \rangle_W, \alpha(S))$ is a Coxeter system and thus the map $\alpha: W \rightarrow \langle \alpha(S) \rangle_W$ is an isomorphism.
\end{proposition}

	\begin{proof}
Let $\alpha$ be a self-similarity of $(W,S)$, put $U := \langle \alpha(S) \rangle_W$ and let $R$ be as in Fact~\ref{fact_for_even}(C) for $(W, S)$ and $U$. By Fact~\ref{fact_for_even}(C) we know that $|R| \leq |\alpha(S)|= |S|$.
We first prove by contradiction that, in fact, equality holds. Indeed, suppose that $|R| < |S|$. Note first that  $\alpha(S)
\subseteq U \cap S^W = R^U$. As $|\alpha(S)| = |S| < |R|$ there exist $t \neq s \in S$ and $r \in R$ such that
$\alpha(s),\alpha(t) \in r^U$. This implies $s^W= \alpha(s)^W = \alpha(t)^W = t^W$ and finally this yields a contradiction
to Fact~\ref{fact_for_even}(D).

\smallskip
\noindent
Thus, $|R| = |S|$ and so there is a bijection $\beta:S \rightarrow R$ such that for each $s \in S$ we have that $\{ \beta(s)\} = s^W \cap R$. Suppose now that the following holds:
\begin{equation}\tag{$\star$} \text{for all $s \neq t \in S$ we have that $o(\beta(s)\beta(t)) = o(st)$.}
\end{equation}
Then we can argue as follows: By $(\star)$ and Fact~\ref{fact_for_even}(E), the map $\beta$ extends to an homomorphism from $W$ to $U$, which is actually an isomorphism, since $(U, R)$ is a Coxeter system.  Consider now the map $\gamma: \alpha \circ \beta^{-1} : R \rightarrow \alpha(S)$. As $(U, R)$ is a Coxeter system, the map $\gamma$ extends to a surjective endomorphism of $U$, but such a map must be an isomorphism, since by Fact~\ref{fact_for_even}(A), $U$ is Hopfian.
Hence, the map $\alpha: W \rightarrow \langle \alpha(S) \rangle_W$ is an isomorphism, if $(\star)$ holds. We show that $(\star)$ holds.

\smallskip
\noindent
 To this extent, let $s \neq t \in S$ be such that $o(st)$ is finite and put $s_1 := \alpha(s)$, $t_1 := \alpha(t)$, and $D_1:= \langle s_1,t_1 \rangle_U \leq U \leq W$.
%
%
Now, by Fact~\ref{fact_for_even}(B) for the Coxeter system $(U, R)$, we can find $r \neq v \in R$ and $u \in U$ such that $o(rv)$ is finite and $D_1^u \leq D_2 :=\langle r, v \rangle_U$. By Fact~\ref{fact_for_even}(B) for the Coxeter system $(W,S)$, we can find $s' \neq t' \in S$ and $w \in W$ such that $o(s't')$ is finite and  $D_2^w \leq \langle s', t' \rangle_W$. Hence, recapitulating, we have the following situation:
$$D_1:= \langle s_1,t_1 \rangle_U, \;\; D_1^u \leq D_2 :=\langle r, v \rangle_U \;\; \text{ and } \;\; D_2^w \leq D_3 :=\langle s', t' \rangle_W.$$

\smallskip
\noindent
We want to show that $\{ s,t \} = \{ s', t' \}$ and $\{ r, v \} = \{ \beta(s), \beta(t) \}$. We show the second equality, the first is proved by an analogous argument. To this extent, let $s_1^u = s^h$ and $t_1^u = t^g$, for $h, g \in W$  (recall that $s_1 = \alpha(s)$, $t_1 = \alpha(t)$, $\alpha(s) \in s^W$ and $\alpha(t) \in t^W$). Now, $s^h \in S^W \cap U = R^U$ and so $s^h \in R^U \cap D^2 = \{ r, v \}^{D_2}$, since $D_2$ is an $R$-parabolic subgroup of $U$. Hence we have that $s^h \in r^{D_2}$ or $s^h \in v^{D_2}$ and not both, by Fact~\ref{fact_for_even}(D). Thus, we have that $r = \beta(s)$ or $v = \beta(s)$, and not both.

\smallskip
\noindent
Similarly, we see that $t^g \in r^{D_2}$ or $t^g \in v^{D_2}$ (and not both, by Fact~\ref{fact_for_even}(D)), and so $r = \beta(t)$ or $v = \beta(t)$, and not both. Suppose now that $r = \beta(s)$ and $r = \beta(t)$, then $\beta$ is not a bijection, a contradiction.  Analogously, it cannot be the case that $v = \beta(s)$ and $v = \beta(t)$. It thus follows that $\{ r, v \} = \{ \beta(s), \beta(t) \}$, as wanted.

\smallskip
\noindent
Hence, putting all together, we actually have the following situation:
$$D_1:= \langle s_1,t_1 \rangle_U, \;\; D_1^u \leq D_2 :=\langle \beta(s), \beta(t) \rangle_U \;\; \text{ and } \;\; D_2^w \leq \langle s, t \rangle_W.$$
Thus, using the fact that $D_2$ and $D_3$ are finite dihedral groups (since $(U, R)$ and $(W, S)$ are Coxeter systems, and $o(\beta(s), \beta(t)), o(s't') < \infty$) we have that:
$$o(st) = o(\alpha(s)\alpha(t)) = o(s_1t_1) = o(s_1^u t_1^u) \leq o(\beta(s)\beta(t)) = o(\beta(s)^w\beta(t)^w) \leq o(st).$$
Hence, $(\star)$ is verified (given that $s$ and $t$ were arbitrary) and the proof is complete.
\end{proof}

	\begin{remark} The following example shows that the even assumption in Proposition~\ref{muller_fact} is necessary. Let $(W, S)$ be the Coxeter system such that $S = \{ s, t, u \}$ and
$o(xy) = 7$ for all $x \neq y \in S$. Put $s' := s$, $t' := t$, $u' := sts$, and let then $S' = \{ s', t', u' \}$. Then $S'$ is a set of
self-similar reflections of $(W, S)$ and $S'$ generates a finite group $G$, and so $G$ is not isomorphic to $W$ (since $W$ is infinite).
\end{remark}

\subsection{Reflection Subgroups of RACGs}



The objective of this subsection is to prove Proposition~\ref{lemma_increasing_subgroups}, which will play a crucial role in Section~\ref{sec_prime_models}, where we will prove in particular that if $W$ is a right-angled Coxeter group of finite rank, then $W$ is a prime model of its theory iff the monoid $Sim(W, S)$ is finitely generated.

\smallskip
\noindent	
 Throughout this subsection $(W,S)$ is a right-angled Coxeter system of finite rank and $S^W:= \{ s^w : s \in S, w \in W \}$ denotes the set of its reflections.
 We shall use the notation of \cite{muhlherr}. This means that we are working with the Cayley graph $\Cay(W,S) = (W, {\bf P})$ of $(W,S)$ where ${\bf P} := \{ \{ w,ws \} \mid w \in W, s \in S \}$. Since $\Cay(W,S)$ is a Coxeter building, we shall use
the language of buildings here. Thus, the vertices of $\Cay(W,S)$ are called {\em chambers}, the edges
are called {\em panels} and a {\em gallery} is a path in $\Cay(W,S)$.
The group $W$ acts by multiplication from the left on $\Cay(W,S)$ and the stabilizer of a panel
$P = \{ w,ws \}$ is the subgroup generated by the reflection $wsw^{-1}$. The set of panels stabilized
by a reflection $t$ is denoted by ${\bf P}(t)$ and the graph $(W, {\bf P} \setminus {\bf P}(t))$
has two connected components which are called the {\em roots associated with $t$}. For a reflection $t \in S^W$ and a chamber $c \in W$ we denote the root associated with $t$
that contains $c$ by $H(t,c)$.
 The set of
{\em chambers lying at the wall of $t$} is:
$${\bf C}(t) := \cup_{P \in {\bf P}(t)} P.$$

	\begin{definition}
	\begin{enumerate}[(1)]
	\item For two chambers $c,d \in W$
we denote their distance in $\Cay(W,S)$ (i.e. the length of a minimal gallery joining them) by $\ell(c,d)$.
	\item For any two nonempty subsets $X,Y$ of $W$ we let:
$$\ell(X,Y) := \min \{ \ell(x,y) \mid x \in X, y \in Y \}$$ and for
$z \in W$, we set $\dist(z,X) := \dist(\{ z \},X)$.
	\item For $t,u \in S^W$ we let the {\em distance} between $t$ and $u$ to be:
$$\dist(t,u) := \ell({\bf C}(t),{\bf C}(u)).$$
\end{enumerate}
\end{definition}

\begin{remark} It is a basic fact that $\dist(t,u)=0$ if $[s,u] = 1$.
\end{remark}

	\begin{fact} Suppose that $o(tu) =\infty$. Then there exists a root (or halfspace) $H$ associated
with $t$ such that ${\bf C}(u) \subseteq H$. This root will be denoted by $H(t,u)$ or $H(t,u,+)$.  Furthermore, $H(t,u,-)$ denotes
the set which does not contain the wall of $u$; equivalently, $H(t,u,-)$ is the (set theoretic) complement of $H(t,u)$ in $W$.
\end{fact}

	\begin{definition}
	\begin{enumerate}[(1)]
	\item A {\em triangle} of $(W,S)$ is a set $T = \{ t_1,t_2,t_3 \} \subseteq S^W$ such that $o(t_it_j) = \infty$ for all $1 \leq i \neq j \leq 3$.
	\item We say that a triangle $T = \{ t_1,t_2,t_3 \}$ is {\em geometric} if $H(t_i,t_j) = H(t_i,t_k)$ whenever $\{ i,j,k \} = \{ 1,2,3 \}$.
	 \item A set $T \subseteq S^W$ of reflections of $(W,S)$ is called {\em geometric} if each triangle contained in $T$ is geometric.
\end{enumerate}
\end{definition}

	\begin{lemma}\label{lemma_geom_refle_minus} Let $t,u,v \in S^W$ be such that $o(tu) = \infty = o(tv)$ and $H(t,u) = H(t,v,-)$.
Then $o(vu) = \infty$.
\end{lemma}

	\begin{proof} This is Assertion (i) of Lemma 2.5 in \cite{muhlherr}.
\end{proof}

	\begin{lemma}\label{lemma_geom_refle} For $T \subseteq S^W$ the following assertions are equivalent:
\begin{enumerate}[(i)]
	\item $T$ is geometric;
	\item if $t,u,v \in T$ are such that $o(tu) = \infty = o(tv)$, then $H(t,u) = H(t,v)$.
\end{enumerate}
\end{lemma}

	\begin{proof} This follows from Lemma~\ref{lemma_geom_refle_minus}.
\end{proof}

\begin{definition}\label{def_fund_domain}
Let $X \subseteq W$. Then we call $X$ {\em convex} if it is a convex subset in the metric space $(W,\ell)$. If $U \leq W$
is a subgroup of $W$, then it acts on $\Cay(W,S)$ by left multiplication and we call $X$ a {\em fundamental domain}
for $U$ when:
\begin{enumerate}[(i)]
	\item $W = \cup_{u \in U} uX$;
	\item $uX \cap X \neq \emptyset$ implies $u = e_U$, for all $u \in U$.
	\end{enumerate}
\end{definition}

	\begin{proposition}\label{prop_mu_2.6} Let $T \subseteq S^W$ be a geometric set of reflections and $U:= \langle T \rangle_W$.
\begin{enumerate}[(a)]
	\item There exists a family of roots $(H_t)_{t \in T}$ such that $H_t$ is a root associated
with $t$ for each $t \in T$ and such that $H_t = H(t,u)$ whenever $t,u \in T$ and $o(tu) = \infty$.
	\item $(U,T)$ is a Coxeter system and if $(H_t)_{t \in T}$ is as in item (a), then
$D := \cap_{t \in T} H_t$ is a convex fundamental domain for the action of $U$ on $W$.
	\item $T^U = U \cap S^W$ and:
	$$T = \{ r \in S^W : |P \cap D| = 1 \mbox{ for some panel $P$ on the wall associated with } r \}.$$
\end{enumerate}
\end{proposition}

	\begin{proof} This is essentially the content of \cite[Proposition~2.6]{muhlherr}.
\end{proof}

	\begin{lemma}\label{lemma_2.6+} Let $T \subseteq S^W$ be a geometric set of reflections, $U = \langle T \rangle_W$, $D$ be as in item (b) of Proposition~\ref{prop_mu_2.6},  $c \in D$ and $r \in S^W \cap U \; (=T^U)$. Then:
\begin{enumerate}[(a)]
	\item If $r$ is not in $T$, then there exists $u \in U$ such that $\ell(c,{\bf C}(uru^{-1})) < \ell(c,{\bf C}(r))$.
	\item For each $t \in T$ we have $\ell(c,{\bf C}(t^u)) \geq \ell(c,{\bf C}(t))$ for all $u \in U$.
	\item In (b) equality holds if and only if $t^u = t$.
\end{enumerate}
\end{lemma}

	\begin{proof} Let $k := \ell(c,{\bf C}(r))$ and let $\gamma = (c=c_0,\ldots,c_k)$ be a minimal gallery from $c$ to ${\bf C}(r)$.
As $r$ is not in $T$ by assumption,
there exists $0 \leq m < k$ such that, for all $0 \leq i \leq m$, $c_i \in D$ and $c_{m+1} \notin D$.
It follows that the unique reflection $u$ associated with the panel $\{ c_m,c_{m+1} \}$ is in $T$.
Now, $\gamma_1 := (c_0,\ldots,c_m) = (u(c_{m+1}),u(c_{m+2}),\ldots,u(c_k))$ (where $u(g)$ is in the sense of Definition~\ref{def_fund_domain}, i.e. $u(g) = ug$) is a gallery from $c$ to ${\bf C}(uru)$
of length $k-1$. This yields item (a).
Finally, parts (b) and (c) follow by an argument similar to the one given in the proof of part (a) by using induction on $\ell_T(u)$,
where $\ell_T$ denotes the length function on $U$ with respect to the generating set $T$. This concludes the proof of the lemma.

\end{proof}

	\begin{proposition}\label{prop_conjugate_geometric} Let $T \subseteq S^W$ be a geometric set of reflections and $U := \langle T \rangle_W$.
If $R$ is a geometric set of reflections such that $U = \langle R \rangle_W$, then
for some $u \in U$:
$$R = T^u = \{ utu^{-1} \mid t \in T \}.$$
\end{proposition}

	\begin{proof} Let $(H_t)_{t \in T}$ and $D := \cap_{t \in T} H_t$ be as in  Proposition~\ref{prop_mu_2.6};
similarly, let $(H_{r})_{r \in R}$ and $E := \cap_{r \in R} H_r$ be as in Proposition~\ref{prop_mu_2.6}.
Then $D$ and $E$ are fundamental domains for the action of $U$ on $W$ by item (b) of Proposition~\ref{prop_mu_2.6}.

	\begin{claim}\label{subclaim} If $D \cap E \neq \emptyset$, then $D = E$ and  $T = R$.
\end{claim}

\begin{claimproof} Let $c \in D \cap E$ and let $d \in D$. Suppose that $d \in D$ is not in $D \cap E$ and that
$\gamma = (c=c_0,\ldots,c_k=d)$ is a minimal gallery from $c$ to $d$. Then $c_i \in D$ for all $0 \leq i \leq k$
because $D$ is convex (cf. Fact~\ref{prop_mu_2.6}(b)).
As $d$ is not in $E$, the gallery $\gamma$ crosses a wall associated to a reflection $r \in R$.
Thus there exists $1 \leq i \leq k$ such that $r(c_i) = c_{i-1}$. As $r \in R \subseteq U$ this yields a contradiction,
because $c_{i-1}$ and $c_i$ are both elements of $D$ and the latter is a fundamental domain for the action of $U$ on $W$.
This shows that $D \subseteq E$ and by symmetry we also have that $E \subseteq D$. Now it follows from item (c) of
Proposition \ref{prop_mu_2.6} that $T=R$.
\end{claimproof}

\smallskip
\noindent
As $D$ and $E$ are fundamental domains for the action of $U$ on $W$ there exists
$u \in U$ such that $u(D) \cap E \neq \emptyset$ and so by the claim above we are done.
\end{proof}

	\begin{definition}\label{def_DU(c)} Let $U \leq W$ be a reflection subgroup of $(W,S)$ and $T := U \cap S^W$.
The graph $\Gamma_U$ is defined by $\Gamma_U := (W,{\bf P} \setminus (\cup_{t \in T} {\bf P}(t)))$.
For each $c \in W$ we define
the set $D_U(c) \subseteq W$ to be the connected component of the graph $\Gamma_U$ containing the chamber $c$.
Furthermore, we set:
$$R_U(c) := \{ r \in S^W \mid |P \cap D_U(c)| = 1 \mbox{ for some  }P \in {\bf P}(r) \}.$$
\end{definition}


	\begin{proposition}\label{hee_prop} Let $U \leq W$ be a reflection subgroup of $(W,S)$, let $c \in W$ be a chamber and put  $D :=D_U(c)$ and $R:=R_U(c)$ (as in Definition~\ref{def_DU(c)}). Then $R$ is a geometric set of reflections, $U=\langle R \rangle_W$ and $D$ is a fundamental
domain for the action of $U$ on $W$ (by left multiplication). Moreover, we have
$U \cap S^W = R^U$.
\end{proposition}

	\begin{proof} This is a geometric version of a result that have been obtained independently
by Deodhar and Dyer (cf. \cite{deo, dyer}). In its present form it is due to H\'ee (cf. \cite{hee}).
\end{proof}

	\begin{corollary} Let $T \subseteq S^W$ and $U = \langle T \rangle_W$. Then the following hold.
\begin{enumerate}[(a)]
\item There is a geometric set $R$ of reflections such that the $\langle R \rangle_W = U$
and all such sets are conjugate in $U$.
\item If $T$ is a finite set and  $R$ is as in item (a), then $|R| \leq |T|$.
\end{enumerate}
\end{corollary}

	\begin{proof} Item (a) is a consequence of the previous proposition and item (b)
is a refinement which follows from Corollary 3.11 in \cite{dyer}.
\end{proof}

	\begin{lemma}\label{lemma_RU} Let $U \leq W$ be a reflection subgroup of $(W,S)$. Then:
\begin{enumerate}[(a)]
\item If $r \in U \cap S^W$ and $c \in {\bf C}(r)$, then $r \in R_U(c)$.
\item Suppose $V \leq W$ is a reflection subgroup of $(W,S)$ such that $V \subseteq U$.
If $r \in V \cap S^W$ and $c \in {\bf C}(r)$, then $r \in R_V(c)$, $r \in R_U(c)$ and $D_U(c) \subseteq D_V(c)$.
\item $D_U(c) = D_V(c)$ if and only if $V = U$.
\end{enumerate}
\end{lemma}

	\begin{proof} Let $P \in {\bf P}(r)$ be the unique panel contained in the wall of $r$ that contains $c$
and let $d = r(c)$. Then $P = \{ c,d \}$ and $d \notin D_U(c)$ because $D_U(c)$ is convex and
$P$ is not an edge of the graph $(W, {\bf P} \setminus (\cup_{t \in U \cap S^W} {\bf P}(t)))$.
It follows that $|P \cap D_U(c)| = 1$ and hence $r \in R_U(c)$. This yields item (a).

\smallskip
\noindent
Concerning item (b), let $r \in V \cap S^W$ and $c \in {\bf C}(r)$. Then it follows from item (a) that $r \in R_V(c)$.
As $V \subseteq U$, we have $S^W \cap V \subseteq S^W \cap U$. Thus $r \in U \cap S^W$ and again by
item (a) we have $r \in R_U(c)$. As $S^W \cap V \subseteq S^W \cap U$ we have
$\cup_{t \in V \cap S^W} {\bf P}(t) \subseteq \cup_{t \in U \cap S^W} {\bf P}(t)$ and
therefore $D_U(c) \subseteq D_V(c)$.

\smallskip
\noindent
Concerning item (c), if $V = U$, then we have $D_U(c) = D_V(c)$ by definition. For the other direction, suppose
$D_U(c) = D_V(c)$, then, again by definition, we have $R_U(c) = R_V(c)$ and
it follows from Proposition~\ref{hee_prop} \mbox{that $V = \langle R_V(c) \rangle_W = \langle R_U(c) \rangle_W = U$.}
\end{proof}

	\begin{proposition}\label{first_prop_self_sim} Let $\alpha$ be a self-similarity of $(W,S)$ and $U := \langle \alpha(S) \rangle_W$. Then:
\begin{enumerate}[(a)]
\item There exists a self-similarity
$\beta$ of $(W,S)$ such that $\langle \beta(S) \rangle_W = U$ and $\beta(S)$ is geometric.
\item If $\gamma$ is a self-similarity such that $\langle \gamma(S) \rangle_W = U$
and $\gamma(S)$ is geometric, then there exists $u \in U$ such that $\beta(S)^u = \gamma(S)$.
\end{enumerate}
\end{proposition}

	\begin{proof} Item (a) can be proved by arguments that are similar to those given in the proof of Proposition~\ref{muller_fact}. Assertion (b) follows from (a) by Proposition~\ref{prop_conjugate_geometric}.
\end{proof}

	\begin{definition}\label{def_complexity_self_similarity} Let $\alpha$ be a self-similarity of $(W,S)$ and let $\beta$ be as in item (a)
of Proposition~\ref{first_prop_self_sim}. The {\em complexity of $\alpha$} is defined to be
the matrix $\Delta(\alpha) := (\dist(\beta(s),\beta(s'))_{s,s' \in S}$ (notice that this is well defined by (b) of Proposition~\ref{first_prop_self_sim}).
For any two square matrices $A = (a_{ss'})_{s,s' \in S}, B=(b_{ss'})_{s,s' \in S}$ with entries
in the natural numbers, we put $A \leq B$ if $a_{ss'} \leq b_{ss'}$ for all $s,s' \in S$;
we put $A < B$ if $A \leq B$ and if there exist $s,s' \in S$ such that $a_{ss'} < b_{ss'}$.
\end{definition}

	\begin{definition} A self-similarity of $(W,S)$ is called {\sl geometric} if $\alpha(S)$ is a geometric set of reflections.
\end{definition}

	\begin{lemma}\label{lemma10} Let $\alpha$ and $\beta$ be geometric self-similarities of $(W,S)$. Let also $U := \langle \alpha(S) \rangle_W$, $V:= \langle \beta(S) \rangle_W$ and suppose that $V \subseteq U$. Then:
\begin{enumerate}[(1)]
	\item $\Delta(\alpha) \leq \Delta(\beta)$;
	\item $\Delta(\alpha) = \Delta(\beta)$ iff $U = V$.
\end{enumerate}
\end{lemma}

	\begin{proof} Note first that $\beta(s) \in V \cap s^W \subseteq U \cap s^W = \alpha(s)^U$
for all $s \in S$.
Let $s \in S$. By conjugating with an element in $U$, we can assume that $\alpha(s) = \beta(s)$, we denote this
reflection by $r$.
Let $s' \in S$, let $\gamma = (c_0,\ldots,c_k)$ be a minimal gallery from ${\bf C}(r)$ to ${\bf C}(\beta(s'))$
and put $c:=c_0$. We set $E:=D_V(c)$ and observe that $r \in R_V(c)$ by item (a) of Lemma~\ref{lemma_RU}.
As $R_V(c)$ and $\beta(S)$ are geometric and $\langle R_V(c) \rangle_W = \langle \beta(S) \rangle_W$, there exists $v \in V$
such that $\beta(S)^v = R_V(c)$ by Proposition~\ref{prop_conjugate_geometric}
and we have $r^v = r$ because $r$ is the only element in $\beta(S)$ and $R_V(c)$ that is in $s^W$.
Similarly, we have $\alpha(S)^u = R_U(c)$ and $r^u = r$ for some $u \in U$.
Thus, we may assume that $\alpha(S) = R_U(c)$ and $\beta(S) = R_V(c)$.
As $\beta(s') \in \alpha(s')^U$, by (b) of Lemma~\ref{lemma_2.6+}:
$$\ell(c,{\bf C}(\alpha(s'))) \leq \ell(c,{\bf C}(\beta(s'))) \leq \ell({\bf C}(\beta(s)), {\bf C}(\beta(s'))) = \dist(\beta(s),\beta(s'))$$
and that equality holds iff $\alpha(s') = \beta(s')$.

\smallskip
\noindent
As $\dist(\alpha(s),\alpha(s')) \leq \dist(c,\alpha(s'))$
this shows that $\Delta(\alpha) \leq \Delta(\beta)$ and that $U = V$ if equality holds.
That $U=V$ implies $\Delta(\alpha) = \Delta(\beta)$ follows from Proposition~\ref{prop_conjugate_geometric}.
\end{proof}

	\begin{lemma} Let $\alpha$ and $\beta$ be self-similarities of $(W,S)$. Then
$\Delta(\alpha \circ \beta) \geq \Delta(\alpha)$ and the inequality is strict if $\beta$ is not an automorphism of $W$.
\end{lemma}

	\begin{proof} Let $U := \langle \alpha(S) \rangle_W$ and $V := \langle \alpha(\beta(S)) \rangle_W$. Then $V \subseteq U$.
Let $R$ be a geometric set of reflections such that $\langle R \rangle_W = U$
and let $T$ be a geometric set of reflections such that $\langle T \rangle_W = V$.
Thus we have unique self-similarities $\alpha',\gamma$ of $(W,S)$ such that $R = \alpha'(S)$
and $\gamma(S) = T$. By definition we have $\Delta(\alpha) = \Delta(\alpha')$ and
$\Delta(\alpha \circ \beta) = \Delta(\gamma)$. As $V \subseteq U$, it follows by Lemma~\ref{lemma10} that $\Delta(\gamma) \geq \Delta(\alpha')$ and that equality holds if and only if $U = V$.
This shows that $\Delta(\alpha \circ \beta) \geq \Delta(\alpha)$ and that
equality holds if and only if $U = V$. Suppose now that $\beta$ is not an automorphism. Since each self-similarity is injective,
it follows that $\beta(W) \neq W$. As $\alpha$ is also injective, it follows
\mbox{that $V = \alpha(\beta(W)) \neq \alpha(W) = U$, and so we are done.}
\end{proof}

	\begin{proposition}\label{lemma_increasing_subgroups} Let $(W, S)$ be a right-angled Coxeter system of finite rank. For every $S$-self-similar subgroup $U$  of $W$ there exists $n < \omega$ such that if $U = V_0 < \cdots < V_\alpha = W$ is a proper chain of $S$-self-similar subgroups of $W$, then $\alpha \leq n$.
\end{proposition}

	\begin{proof} This follows from the fact that the complexity (in the sense of Definition~\ref{def_complexity_self_similarity}) of the corresponding self-similarities decreases strictly along such a chain.
\end{proof}

	\begin{definition} Let $\alpha \in Sim(W, S)$. We say that $\alpha$ is {\em proper} if $\alpha \notin Aut(W)$. We say that $\alpha$ is {\em decomposable} if there are proper $\beta, \gamma \in Sim(W, S)$ such that $\alpha = \beta \circ \gamma$. Finally, $\alpha$ is called irreducible if it is proper and not decomposable.
\end{definition}

	\begin{corollary}\label{product_of_indecomposables_pre} Let $f_1, ..., f_k$ be proper self-similarities, $\alpha = f_1 \circ \cdots \circ f_k$, and let $\Delta(\alpha) = (d_{ss'})_{s, s' \in S}$ be the matrix associated to $\alpha$ from Definition~\ref{def_complexity_self_similarity}, then:
$$k \leq \sum_{s, s' \in S} d_{ss'}.$$
\end{corollary}

\subsection{Word Combinatorics for RACGs}

\begin{definition} Let $(W, S)$ be a right-angled Coxeter system.
	\begin{enumerate}[(1)]
	\item A word $w$ in the alphabet $S$ is a sequence $(s_1, ..., s_k)$ \mbox{with $s_i \neq s_{i+1} \in S$, $i \in [1, k)$.}
	\item We denote words simply as $s_1 \cdots s_k$ instead of $(s_1, ..., s_k)$.
	\item We call each $s_i$ a syllable of the word $s_1 \cdots s_k$.
	\item We say that the word $s_1 \cdots s_k$ spells the element $g \in W$ if $W \models g = s_1 \cdots s_k$.
	\item By convention, the empty word spells the identity element $e$.
	\item The length $\ell(w)$ of the word $w = s_1 \cdots s_n$ is the natural number $n$ (so, in particular, the length of the empty word is $0$).
	\end{enumerate}
\end{definition}

\begin{definition} Let $(W, S)$ be a right-angled Coxeter system.
	\begin{enumerate}[(1)]

	\item We say that the word $w$ is reduced if there is no word with fewer syllables which spells the same element of $W$.
	\item We say that the word $w$ is a normal form for $g \in W$ if $w$ spells $g$ and $w$ is reduced.
	\item We say that the word $w = s_1 \cdots s_k$ is cyclically reduced if $w = e$ or $s_1 \neq s_k$.
	\item We say that $g \in W$ is cyclically reduced if $g$ is spelled by a cyclically reduced word.
\end{enumerate}
\end{definition}

	Notice that the definition of $\ell_S$ from Definition~\ref{reflection_length} is consistent with the following definition of $\ell_S$ (i.e. the one in Definition~\ref{def_length}(ii)).

	\begin{definition}\label{def_length} Let $(W_{\Gamma}, S)$ be a right-angled Coxeter system and let $g \in W_\Gamma$ (so $\Gamma = (S, E_\Gamma)$ is the corresponding Coxeter graph). We define:
	\begin{enumerate}[(1)]
	\item $sp_S(g) = sp(g) = \{ s \in S : s \text{ is a syllable of a normal form for } g \}$;
	\item $\ell_S(g) = \ell(g)= |w|$, for $w$ a normal form for $g$;
	\item $lk_S(g) = lk(g) = \{ s \in S : s E_{\Gamma} t \text{ for every } t \in sp(g) \}$.
\end{enumerate}
\end{definition}



	The following notation is justfied by Fact~\ref{pre_centr_fact}, which is stated soon after.

	\begin{notation}\label{roots_notation} Let $W$ be an irreducible right-angled Coxeter group and $g \in W$.
\begin{enumerate}[(1)]
	\item We denote by $o(g)$ the order of $g$.
	\item If $o(g)$ is infinite, write $g = r^n$ with $1 \leq n < \omega$ maximal.
	\item If $o(g)$ is finite, write $g = r^n$ with $1 \leq n < o(g)$ maximal.
	\item We let $\sqrt{g} = r$, and we call $r$ a root of $g$.
\end{enumerate}
\end{notation}

	\begin{fact}[{\cite[Theorem 3.2]{corr}}]\label{pre_centr_fact} Let $W$ be an irreducible right-angled Coxeter group. Then roots are unique, and so Notation~\ref{roots_notation} is well-defined.
\end{fact}

	Let $W$ be a right-angled Coxeter group of finite rank. Since $W$ decomposes as the direct product of its irreducible components and for $k = hgh^{-1} \in W$, with $g$ cyclically reduced, we have that the centralizer $C_W(k) = hC_W(g)h^{-1}$, in the next theorem we can assume w.l.o.g. that $W$ is irreducible and $g$ is cyclically reduced.

	\begin{fact}[Centralizer Theorem {\cite[Theorem 3.2]{corr}} and \cite{bark}]\label{centralizer} Let $(W, S)$ be an irreducible right-angled Coxeter system. Let $g \in W$ be a cyclically reduced element. Then the centralizer $C_W(g)$ of $g$ in $W$ is $\langle \sqrt{g} \rangle_W \times \langle lk(g)\rangle_W$.
\end{fact}

\begin{fact}[Finite Order Theorem {\cite[Proposition 1.2]{casal1}}]\label{finite_order} Let $(W, S)$ be a right-angled Coxeter system and $k \in W$. Then $k$ has finite order if and only if $k$ has order $2$ if and only if $k = hgh^{-1}$, with $g$ cyclically reduced and $sp(g)$ inducing a clique of $\Gamma$ (i.e. for every $s, t \in sp(g)$ we have $sE_\Gamma t$).
\end{fact}

\subsection{Automorphism Groups of RACGs}\label{auto_sec}

	\begin{fact}[{\cite[Th{\'e}or{\`e}me 2]{castella}}]\label{rigidity} Let $W$ be a right-angled Coxeter group. Let $S$ and $T$ be two Coxeter bases of $W$, then there exists $\alpha \in Aut(W)$ such that $\alpha(S) = T$.
\end{fact}

	A fundamental result of Tits \cite{tits} gives an explicit description of $Aut(W_\Gamma)$ as a semidirect product of two subgroups of $Aut(W_\Gamma)$, namely $Spe(W_\Gamma)$ and $F(\Gamma)$.

	\begin{definition}\label{spe_def} Let $W$ be a right-angled Coxeter group. We denote by $Spe(W)$ the set of automorphisms $\alpha \in Aut(W)$ such that for every involution $h \in W$ there exists $g \in W$ such that:
	$$\alpha(h) = ghg^{-1}.$$
\end{definition}

	\begin{definition}\label{F(Gamma)} Let $\Gamma$ be a graph. We think of the set of finite subsets of $\Gamma$ as a $GF(2)$-vector space (the field with $2$ elements) $V(\Gamma) = (\mathcal{P}_{fin}(\Gamma), \triangle, \cdot)$ by letting:
\begin{enumerate}[(1)]
\item $S_1 \triangle S_2 = (S_1 - S_2) \cup (S_2 - S_1)$;
\item $0 \cdot S = \emptyset$;
\item $1 \cdot S = S$.
\end{enumerate}
We denote by $F(\Gamma)$ the set of linear automorphisms of $V(\Gamma)$ which send finite cliques of $\Gamma$ to finite cliques of $\Gamma$.
\end{definition}

	\begin{remark} Notice that $F(\Gamma)$ is naturally seen as a subgroup of $Aut(W_\Gamma)$ by letting, for $\alpha \in F(\Gamma)$, $\beta_{\alpha}$ be the map such that for every $s \in \Gamma$ we have
	$$\beta_{\alpha}(s) = \prod_{t \in \alpha(s)} t.$$
Abusing notation, we might write $\beta_{\alpha}$ simply as $\alpha$. When we want to stress that $\Gamma = (S, E)$, i.e. we want to make explicit that $S$ is the domain of $\Gamma$, we write $\Gamma$ as $\Gamma_S$. Also, given a basis $S$ of $W$, we denote by $\Gamma_S$ the associated Coxeter graph.
\end{remark}

	\begin{definition} Given a group $G$, a subgroup $H$ of $G$, and a normal subgroup $N$ of $G$, we write $G = N \rtimes H$ when $G = NH$ and $N \cap H = \{e\}$.
\end{definition}

	\begin{fact}[Tits \cite{tits}]\label{Tits} Let $\Gamma$ be a graph. Then: $$Aut(W_\Gamma) = Spe(W_\Gamma) \rtimes F(\Gamma).$$
\end{fact}

\begin{notation}\label{star} Let $\Gamma = (\Gamma, E_\Gamma)$ be a graph. For $v \in \Gamma$, we let:
\begin{enumerate}[(1)]
\item $N(v) = \left\{ v' \in \Gamma : v E_{\Gamma} v' \right\}$;
\item $N^*(v) = N(v) \cup \left\{ v \right\}$.
\end{enumerate}
\end{notation}

	\begin{definition}\label{partial_conj} Let $\Gamma$ be a graph, $s \in \Gamma$ and $C$ a union of connected components of $\Gamma - N^*(s)$. We define an automorphism (cf. Fact \ref{partial_conj_fact}) $\pi_{(s, C)}$ of $W_{\Gamma}$ as follows:
$$\begin{cases} \pi_{(s, C)}(t) = sts \;\;\;\; \text{ if } t \in C \\
			  \pi_{(s, C)}(t) = t \;\;\;\;\;\;\; \text{ otherwise. }
\end{cases} $$
Automorphisms of the form $\pi_{(s, C)}$ are called {\em partial conjugations}.
\end{definition}

	\begin{fact}[\cite{muhlherr}]\label{partial_conj_fact} Let $\Gamma$ be a graph, then the partial conjugations (cf. Def. \ref{partial_conj}) are automorphisms of $W_{\Gamma}$ and, if $\Gamma$ is {\em finite}, then $Spe(W_\Gamma)$ is generated by them.
\end{fact}

	\begin{remark}\label{remark_involutory} Notice that the partial conjugations $\pi_{(s, C)}$ are involutory automorphism, i.e. they have order $2$. Hence, when $\Gamma$ is finite, $Spe(W_\Gamma)$ is generated by finitely  many involutory automorphisms. This will be relevant in Section~\ref{sec_prime_models}.
\end{remark}


	\begin{definition}\label{star_prop} Let $\Gamma$ be a graph.
	\begin{enumerate}[(1)]
	\item We say that $\Gamma$ has the {\em star-property} if for every $v \neq v' \in \Gamma$ we have that $N^*(v) \not\subseteq N^*(v')$ (cf. Notation \ref{star}).
	\item We say that $\Gamma$ is {\em star-connected} if for every $v \in \Gamma$ we have that $\Gamma - N^*(v)$ is connected (cf. Notation \ref{star}).
\end{enumerate}
\end{definition}

	\begin{fact}[{\cite[Commentaire 3]{castella}}]\label{reflection_preserving} Let $\Gamma$ be a graph. The following are equivalent:
	\begin{enumerate}[(1)]
	\item $F(\Gamma) = Aut(\Gamma)$ (cf. Definition \ref{F(Gamma)});
	\item $W_{\Gamma}$ is reflection independent (cf. Definition \ref{def_refl}));
	\item $\Gamma$ has the star-property (cf. Definition \ref{star_prop}(1)).
\end{enumerate}
\end{fact}

\begin{fact}[{\cite[Comm. 3]{castella}}]\label{spe_inn} Let $\Gamma$ be a {\em finite} graph. {\mbox The following are equivalent:}
	\begin{enumerate}[(1)]
	\item $Spe(W_{\Gamma}) = Inn(W_{\Gamma})$ (cf. Definition \ref{spe_def});
	\item $\Gamma$ is star-connected (cf. Definition \ref{star_prop}(2)).
\end{enumerate}
\end{fact}

\subsection{$2$-Spherical Coxeter Groups}

	\begin{definition}[\cite{reflec_abc_cox}]\label{def_finite_cont} Let $(W, S)$ be a Coxeter system of finite rank and let $w \in W$ be of finite order. We define the finite continuation of $w$, denoted as $FC(w)$, to be the intersection of all the maximal finite subgroups of $W$ containing $w$.
\end{definition}

	\begin{fact}[{\cite[Lemma~9.3]{intrinsic_reflections}}]\label{fact_finite_cont} Let $(W, S)$ be a Coxeter system of finite rank and let $w \in W$ be of finite order. Then $FC(w)$ is well-defined and it is the intersection of all the maximal spherical subgroups of $W$ containing $w$.
\end{fact}

	\begin{fact}[{\cite[Main Result and Theorem~1]{reflec_abc_cox}}]\label{fact2_finite_cont} Let $(W, S)$ be an irreducible, infinite, $2$-spherical Coxeter system of finite rank. Then $W$ is reflection-independent and the set of reflection is exactly the set of involutions of $W$ such that $FC(w) = \{ e, w \}$. Furthermore, if $R \subseteq W$ is such that $(W, R$) is a Coxeter system, then there exists $w \in W$ such that $R^w = S$ (i.e., in the terminology of \cite{reflec_abc_cox},  $W$ is strongly rigid).
\end{fact}

\section{Superstability in Coxeter Groups}\label{sec_superst}

	In this section we prove Theorems~\ref{th_char_sstab} and \ref{pure_theorem}, and their corollaries. Section~\ref{sec_negative_side} will be concerned with sufficient conditions for unsuperstability for a given group $G$, while in Section~\ref{sec_positive_side} we will prove that affine Coxeter groups are superstable.

\subsection{The Negative Side}\label{sec_negative_side}

	\begin{notation}\label{notation_strings} Given $(y_i : i < \omega)$ and $n < \omega$, let $\bar{y}_{[n)}= (y_i : i < n)$. Also, given $\{i_0, ..., i_{k-1} \} = I \subseteq \{ 0, ..., n-1 \}$ we let $\bar{y}_{I} = (y_{i_\ell} : \ell < k)$. The consistency of the two notations is given letting $\{0, ..., n-1 \} = [0, n) = [n)$.
\end{notation}

	\begin{notation} Given $\eta, \nu \in \omega^{<\omega}$, we write $\nu \triangleleft\eta$ if $\eta$ extends $\nu$, and $\nu \trianglelefteq \eta$ if $\eta$ extends $\nu$ or $\eta = \nu$. Also, we identify the number $n < \omega$ with the set $\{0, ..., n-1\}$.
\end{notation}

	\begin{definition}\label{def_unsuper} We say that the first-order theory $T$ is {\em not} superstable if:
\begin{enumerate}[(a)]
	\item there are formulas $\varphi_n(x, \bar{y}_{[k_n)})$, for $n < \omega$ and $k_n = k(n) \geq n$;
	\item there is $M \models T$;
	\item\label{c} there are $b_{\eta} \in M$, for $\eta \in \omega^{<\omega}$;
	\item\label{d} there are $\bar{a}_\nu \in M^{k(n)}$, for $\nu \in \omega^n$;
	\item for $\nu \in \omega^n$ and $\eta \in \omega^{<\omega}$, $M \models \varphi_n(b_{\eta}, \bar{a}_\nu)$, if $\nu \triangleleft \eta$;
	\item\label{f} there is $m(n) < \omega$ such that if $\nu \in \omega^n$, $k, j < \omega$, and $\eta = \nu^\frown (k, j)$, then:
	$$|\{i < \omega: M \models \varphi_{n+1}(b_\eta, \bar{a}_{\nu^{\frown} (i)}) \} | \leq m(n).$$
\end{enumerate}
\end{definition}

	There are many equivalent definitions of superstability, we use the above for convenience, a more easy to understand definition of superstability uses types: $T$ is said to be superstable if it is $\kappa$-stable for every $\kappa \geq 2^{\aleph_0}$ (see e.g. \cite[pg. 172]{marker}), where a theory $T$ is said to be $\kappa$-stable when for every $M \models T$ and $A \subseteq M$ with $|A| = \kappa$ we have that the number of finitary types over $A$ is of size $\kappa$ (see e.g. \cite[pg. 135]{marker}). A canonical example of a superstable structure is the abelian group~$\mathbb{Z}$.

	\begin{remark}\label{rmk_superst} {\rm In Definition~\ref{def_unsuper}(\ref{f}) we can take $m(n) = 1$, for every $n < \omega$.
\newline [Why? Without loss of generality $M$ is $\aleph_1$-saturated, and so it is enough to find $(\bar{a}_\eta, b_\eta : \eta \in \omega^{ \leq n+2})$, for every $n < \omega$. To this extent, let $\nu \in \omega^{n}$ and consider the function $f_\nu: \omega^3 \rightarrow \{ 0, 1\}$ such that $f_\nu(k, j, i) = 1$, if $k = i$ or $M \not \models \varphi_{n+1}(b_{\nu^{\frown} (k, j)}, \bar{a}_{\nu^{\frown} (i)})$, and $f_\nu(k, j, i) = 0$ otherwise. By the Infinite Ramsey Theorem there is an infinite $f$-homogeneous subset of $\omega$, and by clause (\ref{f}) of Definition~\ref{def_unsuper} this set has to have color $1$, and so we can conclude easily.]}
\end{remark}

	\begin{remark}\label{rmk_superst+} {\rm In Definition~\ref{def_unsuper}(\ref{c}-\ref{d}) we can restrict to $\eta$'s and $\nu$'s in the set:
$$\mathrm{inc}_{<\omega}(\omega) = \{ \sigma \in \omega^{<\omega} : \sigma(0) < \sigma (1) < \cdots < \sigma(|\sigma|-1) \}.$$
We denote the subset of $\mathrm{inc}_{<\omega}(\omega)$ consisting of the sequences with $\{0, ..., m-1\}$ as domain by $\mathrm{inc}_{m}(\omega)$ -- this notation will be used in the proof of Theorem~\ref{general_criterion}.}
\end{remark}

	\begin{theorem}\label{general_criterion} Let $G$ be a group. A sufficient condition for the unsuperstability of $G$ is that there is a subgroup $H \subseteq G$ (not necessarily definable), $2 \leq n < \omega$, and $1 \leq k < \omega$ such that:
	\begin{enumerate}[(a)]
	\item if $a \in H - \{ e_G \}$, then $X_a := \{ x \in G : x^n = a\} \subseteq H - \{e_G\}$, and $|X_a| \leq k$;
	\item there are $a_\ell \in H - \{e_G \}$, for $\ell < \omega$, such that:
	\begin{enumerate}[(i)]
	\item  $\ell_1 \neq \ell_2$ implies $a^{-1}_{\ell_1} a_{\ell_2} \notin \{ x^ny^n : x, y \in H \}$;
	\item for every $(s_1, ..., s_k) \in \omega^{<\omega}$ we have $a_{s_1} \cdots a_{s_k} \neq e_G$.
\end{enumerate}
\end{enumerate}
\end{theorem}

	\begin{proof} Let $n$ and $k$ be as in the statement of the theorem. By induction on $m < \omega$, we define the group word $w_m(z, y_{[m)})$ (recall Notation~\ref{notation_strings}) as follows:
	\begin{enumerate}[(i)]
	\item $w_0(z) = z$;
	\item $w_{m+1}(z, \bar{y}_{[m+1)}) = w_m(y_mz^n, \bar{y}_{[m)})$.
	\end{enumerate}
	Notice that if $m < \omega$ is such that $m = m_1 + m_2$, then:
	$$w_m(x, \bar{y}_{[m)}) = w_{m_2}(w_{m_1}(x, \bar{y}_{[m_1)}), \bar{y}_{[m_1, m)}).$$
Let now $\varphi_m(x, \bar{y}_{[m)})$ be the formula:
	\begin{equation}\label{equation1} \exists z(x = w_m(z, \bar{y}_{[m)})\wedge \bigwedge_{\ell \leq m} (w_\ell(z, \bar{y}_{[m-\ell, m)}))^n \neq e),
\tag*{$(\star)_1$}  \end{equation}
(clearly for $m = 0$, the set $[m, m)$ is simply $\emptyset$ and so $w_0(z, y_{[m, m)}) = w_0(z) = z$.)
	Notice that:
	\begin{equation}\label{equation2}
	\varphi_{m+1}(x, \bar{y}_{[m+1)}) \vdash \varphi_{m}(x, \bar{y}_{[m)}) \vdash \cdots \tag*{$(\star)_2$}  \end{equation}
	We claim that $(\varphi_{m}(x, \bar{y}_{[m)}) : m < \omega)$ is a witness for the unsuperstability of $G$, referring here to Definition~\ref{def_unsuper} (cf. also Remarks~\ref{rmk_superst}~and~\ref{rmk_superst+}).

\smallskip
\noindent Let $a_\ell \in H - \{e_G \}$, for $\ell < \omega$, be as in the statement of the theorem. Now, for $m < \omega$ and $\nu \in \mathrm{inc}_m(\omega)$ (cf. Remark~\ref{rmk_superst+}), let:
	\begin{equation}\label{equation3} \bar{a}_\nu = (a_{\nu(\ell)} : \ell < m) \in H^m
\tag*{$(\star)_3$}. \end{equation}
For $m < \omega$ and $\eta \in \mathrm{inc}_{m+1}(\omega)$, let:
\begin{equation}\label{equation4} b_\eta = w_m(a_{\eta(m)}, \bar{a}_{\eta \restriction m}) \in H - \{e_G\} \;\; (\text{by clause (b)(ii)})
\tag*{$(\star)_4$}. \end{equation}
Clearly clauses (a)-(d) of Definition~\ref{def_unsuper} hold. Furthermore, we have:
\begin{equation}\label{equation5}
\eta \in \mathrm{inc}_{m+1}(\omega) \;\; \Rightarrow \;\; M \models \varphi_m(b_\eta, \bar{a}_{\eta\restriction m}).
\tag*{$(\star)_5$} \end{equation}
[Why? The satisfaction of the first conjunct of the formula $\varphi_m(b_\eta, \bar{a}_{\eta\restriction m})$ is ensured by the choice of $b_{\eta}$, while the second conjunct is by clause (b)(ii) of the theorem.]
More strongly, we have:
\begin{equation}\label{equation6} \eta \in \mathrm{inc}_{k}(\omega), \; \nu \triangleleft \eta \;\; \Rightarrow \;\; M \models \varphi_{|\nu|}(b_\eta, \bar{a}_{\nu}).
\tag*{$(\star)_6$} \end{equation}
[Why? By $(\star)_2$ and $(\star)_5$.]
\newline Further, we have that if $b \in H - \{e_G\}$ and $\eta \in \mathrm{inc}_{m}(\omega)$, then letting:
\begin{equation}\label{equation7} C^\eta_b := \{ c \in G : G \models b = w_{m}(c, \bar{a}_{\eta}) \wedge \bigwedge_{\ell \leq m} (w_\ell(c, \bar{y}_{[m-\ell, m)}))^n \neq e \},
\tag*{$(\star)_7$} \end{equation}
we have:
\begin{equation}\label{equation8} C^\eta_b \subseteq H - \{e_G\}\; \text{ and } \; |C^\eta_b| \leq k^m.
\tag*{$(\star)_8$} \end{equation}
[Why? We prove this by induction on $m = |\eta|$ using clause (a) of the theorem. For $m = 0$ this is obvious. For $m = \ell +1$, let $\nu = \eta \restriction \ell$. Recall that by inductive hypothesis we have that $|C^\nu_b| \leq k^\ell$. Further, clearly, $\eta \in \mathrm{inc}_{\ell+1}(\omega)$. Now, if $G \models b = w_{\ell+1}(c, \bar{a}_\eta)$, then, letting $d_c = a_{\eta(\ell)}c^n$, we have $G \models b = w_\ell(d_c, a_{\eta \restriction \ell})$, and thus $d_c \in C^\nu_\ell \subseteq H - \{ e_G \}$ (by inductive hypothesis). That is, we have a function $c \mapsto d_c$ from $C^\eta_b$ into $C^\nu_b$. Hence, it suffices to prove that for each $d \in C^\nu_{b}$ we have:
$$\mathcal{D}_d := \{ c \in C^\eta_b : d_c = d\} \subseteq H - \{e_G\} \; \text{ and } \; |\mathcal{D}_d| \leq k,$$
since then we would have:
	$$|C^\eta_b| \leq |C^\nu_b|k \leq k^\ell k = k^{\ell+1} = k^m.$$
Let then $d \in C^\nu_{b}$ and $c \in \mathcal{D}_d$. Since $a_{\eta(\ell)}c^n = d_c = d\in H - \{ e_G \}$ we have that $c^n = a^{-1}_{\eta(\ell)}d_c \in H$ (recall that $H$ is a subgroup), and by the choice of $c$ we have that $c^n \neq e_G$ (cf. $(\star)_7$, second conjunct, $m=0$). Thus, by clause (a) of the theorem, we have that $c \in H - \{ e_G \}$.
Thus, $\mathcal{D}_d \subseteq \{ x \in G : x^n = a^{-1}_{\eta(\ell)}d \} = X_d$, and by clause~(a) of the statement of the theorem we have that $|X_d| \leq k$, and so $|\mathcal{D}_d| \leq k$.]
\newline Finally, we have that if $\nu \in \omega^m$, $k, j < \omega$, and $\eta = \nu^\frown (k, j)$, then:
\begin{equation}\label{equation9}
|\{i < \omega: G \models \varphi_{m+1}(b_\eta, \bar{a}_{\nu^{\frown} (i)}) \} | \leq k^m.
\tag*{$(\star)_9$} \end{equation}
[Why? Let $\mathcal{U}^m_\eta := \{i < \omega: G \models \varphi_{m+1}(b_\eta, \bar{a}_{\nu^{\frown} (i)}) \}$, and for every $i \in \mathcal{U}^m_\eta$, choose $c_i \in G$ such that:
$$G \models b_\eta = w_{m+1}(c_i, \bar{a}_{\nu^{\frown} (i)}) \wedge \bigwedge_{\ell \leq m+1} (w_\ell(c_i, \bar{y}_{[(m+1)-\ell, m+1)}))^n \neq e.$$
Note that, for $i \in \mathcal{U}^m_\eta$, $c_i \in H - \{e_G\}$, since $b_\eta \neq e_G$ (cf. \ref{equation4}) and $c_i \in C^{\nu^{\frown} (i)}_{b_\eta}$, and, by $(\star)_7$, we have that  $C^{\nu^{\frown} (i)}_{b_\eta} \subseteq H - \{ e_G \}$. Further, for $i \in \mathcal{U}^m_\eta$, we have:
$$G \models b_\eta = w_{m}(a_{i}c_i^n, \bar{a}_\nu),$$
and so $a_{i}c_i^n \in C^\nu_{b_\eta}$ (recall that $\nu^{\frown} (i)(m) = i$). For the sake of contradiction, assume that $|\mathcal{U}^m_\eta| > k^m$. By $(\star)_7$ we have that
$|C^\nu_{b_\eta}| \leq k^m$, and so for some $i \neq j \in \mathcal{U}^m_\eta$ we have $a_ic_i^n = a_jc_j^n$. Thus, we have:
\begin{equation}\label{equation10}
a_j^{-1}a_i = (c_j)^n(c^{-1}_i)^n.
\tag*{$(\star)_{10}$} \end{equation}
But then, since $c_i, c_j \in H$ (as observed above), the conclusion $(\star)_{10}$ is in contradiction with clause (a)(i) of the statement of the theorem. Thus, $(\star)_{9}$ holds.]
\newline Hence, by \ref{equation6} and \ref{equation9}, conditions (e) and (f) from Def.~\ref{def_unsuper} are also satisfied.
\end{proof}

	\begin{lemma}\label{negative_side} Let $(W, S)$ be a Coxeter system of finite rank which is not non-affine $2$-spherical, and assume that $W$ is infinite. If $(W, S)$ fails the condition of Theorem~\ref{th_char_sstab}, then $W$ is not superstable (cf. Definition~\ref{def_unsuper}).
\end{lemma}

	\begin{proof} Let $(W, S)$ be of finite rank, irreducible,  infinite, not of affine type and not $2$-spherical. By Theorem~\ref{general_criterion} it suffices to find a non-abelian free subgroup $\mathbb{F} \leq W$ and $n < \omega$ such that for every $x \in W$, if $x^n \in \mathbb{F}$, then $x \in \mathbb{F}$. Now, since $W$ is not $2$-spherical and not affine, we can find a special parabolic subgroup $P$ of $W$ of rank $3$ such that its associated graph contains a non-edge. Then, by \cite[Theorem~1]{gordon}, we know that $P$ is virtually a non-abelian free group. Let then $\mathbb{F} \leq P$ be such that the index $[P : \mathbb{F}] = t < \omega$. Let now $n$ be a prime number bigger than $c$, where:
$$c = max\{ t, max\{ |B| : \text{$B$ a spherical special $S$-parabolic subgroup of $W$} \} \}.$$
(Recall that $W$ is of finite rank and so this $c$ is well-defined.) Without loss of generality we can assume that $\mathbb{F}$ is normal in $P$ (if not, replace $\mathbb{F}$ with $\bigcap \{ g \mathbb{F}g^{-1} : g \in P \}$, which is still non-abelian free and of finite index in $P$). We claim that $n$ is as wanted, i.e. for every $x \in W$, if $x^n \in \mathbb{F}$, then $x \in \mathbb{F}$. To this extent, let $w \in W$ be such that $w^n \in \mathbb{F}$. First of all we claim that $w \in P$. Since $w^n \in P$, we have that $w^n \in N_W(P) = C_W(P) \times P$ (see e.g. \cite[Lemma~2.2]{caprace_normalizer}), Let then $\pi$ be the canonical homomorphism mapping $N_W(P)$ onto $C_W(P)$. Then obviously $\pi(w)^n = e$, and so $\pi(w) = e$, since the order of $\pi(w)$ divides $n$, which is a prime number, and $n$ is bigger than all the orders of finite elements from $W$, by the choice of $n$. Hence, by the nature of $\pi$, we can conclude that $w \in P$, as claimed above. We now show that $w$ is actually in $\mathbb{F}$. Since by assumption $w^n \in \mathbb{F}$ we have that $P/\mathbb{F} \models (wF)^n = e$ (recall that we assume that $\mathbb{F}$ is normal in $P$), but then the order of $wF$ in $P/\mathbb{F}$ divides a prime number which is bigger than all the orders of elements from $P/\mathbb{F}$ (since $n > [P : \mathbb{F}]$). Hence, $P/\mathbb{F} \models wF = e$, that is $w \in \mathbb{F}$, as wanted.
\end{proof}

	\begin{proof}[Proof of Theorem~\ref{pure_theorem}] This is by  Theorem~\ref{general_criterion} and properties of free groups. The claim about virtually non-abelian free groups is proved as in the proof of Lemma~\ref{negative_side}.
\end{proof}

	\begin{proof}[Proof of Corollary~\ref{Artin_theorem}] This follows from Theorem \ref{pure_theorem}.
\end{proof}

\subsection{The Positive Side}\label{sec_positive_side}

	\begin{fact}[{\cite[Proposition~2, pg. 146]{brown}}]\label{fact_affine} Let $(W, S)$ be an irreducible affine Coxeter group.
Then there exists $N \trianglelefteq W$ and $W_0 \leq W$ such that:
\begin{enumerate}[(1)]
	\item $W = N \rtimes W_0$;
	\item $N \cong \mathbb{Z}^d$, for some $1 \leq d < \omega$;
	\item $W_0$ is a Weyl group (and so, in particular, a finite Coxeter group).
\end{enumerate}
\end{fact}

	Thus, in light of Fact~\ref{fact_affine}, we show:

	\begin{proposition}\label{prop_semidirect} Let $Q$ be a finite group, $1 \leq d < \omega$, and $\theta: Q \rightarrow Aut(\mathbb{Z}^d)$ an homomorphism. Then the semidirect product $G :=\mathbb{Z}^d \rtimes_\theta Q$ is interpretable in $\mathbb{Z}$ with finitely many parameters, and so, in particular, the group $G$ is superstable.
\end{proposition}

	We prove two lemmas from which Proposition~\ref{prop_semidirect} follows.

	\begin{lemma}\label{lemma_inter} In the context of Proposition~\ref{prop_semidirect}, the group $G$ is interpretable with finitely many parameters in the structure $M = (\mathbb{Z}^d, +,  \pi_x)_{x \in Q}$, where $\pi_x := \theta(x)$. 
\end{lemma}

	\begin{proof} Let $G = (G, \cdot_G) = (\mathbb{Z}^d \times Q, \cdot_G)$, and enumerate $\{ \pi_x : x \in Q \}$ as $(t_0, ..., t_{n-1})$. Let also $\mathbb{Z}^d = N$. We represent $N \times Q$ as the collection of pairs $(a, \hat{t}_i)$ with $a \in N$ and $\hat{t}_i =(i, \underbrace{0, ..., 0}_{d-1})$, for $i < n$, thus using the elements $\{ (i, \underbrace{0, ..., 0}_{d-1}) : i < n \}$ as parameters. For the rest of the proof we do not distinguish between the elements $t_i$ and $\hat{t}_i$.
We are left to show that the product:
	$$(a_1, t_1) \cdot_G (a_2, t_2) = (a_3, t_3) = (a_1t_1(a_2), t_1t_2) \in N \times Q,$$
is definable in $M$.  Now, $Q \models t_3 = t_1t_2$ is clearly definable in $M$, since $Q$ is finite. We then conclude observing that $N \models a_3 = a_1t_1(a_2)$ is also definable in $M$, because the function symbols $t_\ell$, for $\ell < n$, are part of the signature of the structure~$M$.
\end{proof}

	\begin{fact}\label{matrix} Let $1 \leq d < \omega$, then $Aut(\mathbb{Z}^d)$ is the group of $d \times d$ invertible $\mathbb{Z}$-matrices.
\end{fact}

	\begin{lemma}\label{lemma_sstab} The structure $M = (\mathbb{Z}^d, +,  \pi_x)_{x \in Q}$ from Lemma~\ref{lemma_inter} is interpretable with finitely many parameters in the abelian group $\mathbb{Z}$.
\end{lemma}

	\begin{proof} Let $(\pi_0, ..., \pi_{n-1})$ enumerate $\{ \pi_x : x \in Q\}$. By Fact~\ref{matrix}, for every $\ell < n$, $\pi_\ell$ can be represented as the invertible $\mathbb{Z}$-matrix $(a^\ell_{i, j})_{i, j < d}$. We now show that we can interpret $M$ in $\mathbb{Z}$ with the set of parameters $A = \{a^\ell_{i, j} : \ell < n, i < d, j < d \}$. The domain of $M$ under the interpretation is naturally the set $\mathbb{Z}^d$. The additive group structure of $M$ is defined coordinate-wise. We are then left to show that the function symbols $\pi_\ell$'s are definable in $\mathbb{Z}$ with parameters from $A$. To this extent, let $\ell < n$, and $\bar{b} = (b_0, ..., b_{d-1}) \in \mathbb{Z}^d$. Then, letting $(c_0, ..., c_{d-1}) = \pi_\ell(\bar{b})$, we have:
	$$c_i = \sum_{j < d} a^\ell_{i, j} b_j,$$
which is clearly definable in $\mathbb{Z}$ over $A$.
\end{proof}

	\begin{proof}[Proof of Proposition~\ref{prop_semidirect}] By Lemmas~\ref{lemma_inter} and \ref{lemma_sstab} (the fact that the interpretation uses finitely parameters is not a problem, see e.g. \cite[pg. 287]{poizat_book}).
\end{proof}

	\begin{proof}[Proof of Theorem~\ref{th_char_sstab}] The fact that the condition is necessary is by Lemma~\ref{negative_side}. The fact that the condition is sufficient is by Fact~\ref{fact_affine}, Proposition~\ref{prop_semidirect} and the fact that finite direct products of superstable groups are superstable.
\end{proof}

	\begin{proof}[Proof of Theorem~\ref{Artin_theorem}] If the right-angled Artin group $A$ is not abelian, then argue as in Lemma~\ref{negative_side}. We are then left with the \mbox{case $A \cong \bigoplus_{\beta < \alpha}\mathbb{Z}$, and so we are done.}
\end{proof}

	\begin{proof}[Proof of Theorem~\ref{decidability_affine_Coxeter}] The only thing which is left to show is the decidability claim, but this is clear simply observing the two following facts:
	\begin{enumerate}[(1)]
	\item any expansion of the abelian group $\mathbb{Z}$ with finitely many constants is decidable (this structure is definable in $(\mathbb{Z}, +, 1, 0)$, which is well-known to be decidable);
	\item if a structure $M$ is $\emptyset$-intepretable into a structure $N$, then the decidability of $Th(N)$ implies the decidability of $Th(M)$.
\end{enumerate}	
\end{proof}

\section{Which Coxeter Groups are Domains?}\label{sec_domains}

	In this section we prove Theorem~\ref{domain_conj}.

\subsection{Preparatory Work}

	In this section we lay the preparatory work towards a proof of Theorem~\ref{domain_conj}. We invite the reader to recall the terminology from Section~\ref{preliminaries_sec}. Also, given a group $G$ and $A, B \subseteq G$ we let $[A, B] = \{a^{-1}b^{-1}ab : a \in A, b \in B\}$.
	
	\begin{lemma} \label{normalparabolic} Let $W$ be a Coxeter group of finite rank.
Let $N$ be a normal subgroup of $W$. Then $Pc_S(N) = \langle J \rangle_W$ for some $J \subseteq S$ such that $[J,S - J] = e$.
\end{lemma}

	\begin{proof} Let $M := Pc_S(N)$. As $W \leq N_W(N)$ we have $W \leq N_W(M)$ by Item (c) of Lemma \ref{parabclosure},
and therefore $M$ is a normal subgroup of $W$. As $M$ is an $S$-parabolic subgroup of $W$, we have
$M = w \langle J \rangle_W w^{-1}$ for some $w \in W$ and some $J \subseteq S$. As $M$ is a normal subgroup of $W$
we have $M = w^{-1}Mw = \langle J \rangle_W$. As $M$ is normal, we have $t\langle J \rangle_W t = \langle J \rangle_W$
for each $t \in S$ and so $[S - J,J] = e$ by Lemma~\ref{basicsonJ}(c).
\end{proof}

\begin{lemma} \label{MPWlemma} Let $(W, S)$ be a Coxeter system of finite rank.
Let $t \in S$, let $K \subseteq S$ be the irreducible component of $(W,S)$ containing $t$, let $J := S - \{ t \}$
and let $U := N_W(\langle J \rangle_W)$. If $U \neq \langle J \rangle_W$, then $K$ is a spherical subset of $S$.
In particular, if $(W,S)$ is irreducible and $U \neq \langle J \rangle_W$, then $(W,S)$ is spherical.
\end{lemma}

\begin{proof} We put $\Gamma := \langle J \rangle_W$ and remark that $\Gamma \leq U$. The group $\Gamma$
acts on the Coxeter building $\Sigma(W,S)$ and $R := \langle J \rangle_W \subseteq W$ is a $\Gamma$-chamber
(in the sense of \cite[Definition~22.2]{descent}. As $U = N_W(\Gamma)$, the group $U$ acts on the set of $\Gamma$-chambers.
If $U \neq \Gamma = Stab_W(R)$, then there exists a $\Gamma$-chamber $T \neq R$. By \cite[Proposition~21.3]{descent},
$T$ is parallel to $R$ (in the sense of \cite[Definition~21.7]{descent}), and therefore it follows from \cite[Proposition~21.50]{descent} that $K$ is spherical.
\end{proof}

\begin{lemma} \label{normalsubgroups}
Let $(W,S)$ be a Coxeter system of finite rank, and suppose that $(W, S)$ is irreducible and non-spherical. Let $N$ be a normal subgroup of $W$ and let $J \subseteq S$
be such that $\emptyset \neq J \neq S$. If $N \leq N_W(\langle J \rangle_W)$, then $|N| = 1$.
\end{lemma}

\begin{proof} We proceed by induction on $k:= |S - J|$.

\smallskip
\noindent
If $k=1$, then $N \leq \langle J \rangle_W$ by Lemma \ref{MPWlemma} and therefore $M := Pc_S(N) \neq W$.
Since $(W,S)$ is irreducible, it follows from Lemma \ref{normalparabolic} that $|M| = 1$ and hence $|N| = 1$.

\smallskip
\noindent
Suppose $k >1$. As $(W,S)$ is irreducible, there exists $t \in S - J$ such that $[t,J] \neq e$
and it follows that $t\langle J \rangle_W t \neq \langle J \rangle_W$ by Item (c) of Lemma \ref{basicsonJ}.
As $N$ is normal in $W$ we have $tNt = N$ and therefore $N \leq N_W(t \langle J \rangle_W t)$. Let $U$ be the
subgroup of $W$ generated by $\langle J \rangle_W$ and $t \langle J \rangle_W t$. Then $N \leq N_W(U)$.
Furthermore, $U$ properly contains $\langle J \rangle_W$ and it is itself contained in $\langle K \rangle_W$
where $K := J \cup \{ t \}$ and therefore $Pc_S(U) = \langle K \rangle_W$. Thus $N$ normalizes $\langle K \rangle_W$
by Item (c) of Lemma \ref{parabclosure} and therefore it follows by induction that $|N|=1$, as wanted.
\end{proof}

\subsection{A Proof of Theorem~\ref{domain_conj}}




In this section we prove Theorem~\ref{domain_conj}. We shall first need the following important result of Daan Krammer.

	\begin{theorem} \label{Krammerthm} Let $(W,S)$ be an irreducible, non-spherical Coxeter system of finite rank. Suppose that there exists a subgroup $H$ of $W$ such that $Pc_S(H) = W$ and such that $H$ is free abelian of rank 2.
Then $(W,S)$ is affine.
\end{theorem}

	\begin{proof}
This follows from \cite[Theorem~6.8.2]{krammer}.
\end{proof}

\begin{lemma} \label{infiniteorder} Let $(W,S)$ be an irreducible, non-spherical Coxeter system of finite rank.
Let $x,y \in W$ be such that $x \neq e \neq y$ and $[x,y^w]=e$ for all $w \in W$.
Then $Pc_S(x) = W = Pc_S(y)$; moreover, $x$ and $y$ have both infinite order.
\end{lemma}

\begin{proof} We put $N := \langle y^w : w \in W \rangle_W$ and observe that
$N$ is a normal subgroup of $W$ such that $|N| \neq 1$. As $y^w \in C_W(x)$
for each $w \in W$ we have $N \leq C_W(x)$. Let $P := Pc_S(x)$. Then $N \leq C_W(x) \leq N_W(P)$, by Lemma~\ref{parabclosure}(c).

\smallskip
\noindent
We have $P = w \langle J \rangle_W w^{-1}$ for some $w \in W$ and some $J \subseteq S$.
As $e \neq x \in P$, we have $J \neq \emptyset$. Furthermore $N = w^{-1}Nw$ normalizes $\langle J \rangle_W$.
Since $|N| \neq 1$ it follows by Lemma~\ref{normalsubgroups} that $J = S$ and hence $Pc_S(x) = W$.
Finally, it follows from Item (b) of Lemma~\ref{parabclosure} (and the fact that
$(W,S)$ is non-spherical) \mbox{that $x$ has infinite order.}

\smallskip
\noindent
As $[x,y^w] = e$ for all $w \in W$, we have also $[y,x^w]=e$ for all $w \in W$. Thus, it follows that
$Pc_S(y) = W$ and that $y$ has infinite order as well.
\end{proof}

\begin{lemma} \label{centerofnormal}  Let $(W,S)$ be an irreducible, non-spherical Coxeter system of finite rank.
If there exists a normal subgroup $N$ of $W$ having a non-trivial center, then $(W,S)$ is affine.
\end{lemma}

\begin{proof} Let $N$ be a normal subgroup of $W$, let $Z$ be the center of $N$ and let $e \neq z \in Z$.
Then $Z$ is normal in $W$ and we have $[z,z^w]=e$ for all $w \in W$. By Lemma \ref{infiniteorder}
we have $Pc_S(z) = W$ and that $z$ has infinite order.

\smallskip
\noindent
For each $s \in S$ we have $szs \in Z$.
Suppose first that $szs \in \{ z,z^{-1} \}$ for all $s \in S$.
Then $T := \langle z \rangle_W$ is a normal subgroup of $W$.
Assume, by contradiction, that there exists $s \in S$ such that $szs = z$. Then $[s,z^w] = e$ for all $w \in W$.
As $e \neq s \in W$ is of finite order, Lemma \ref{infiniteorder} yields a contradiction. Thus we have $szs = z^{-1}$
for all $s \in S$. Let $t \neq s$ be elements of $S$. Then $e \neq ts \in C_W(T)$ and therefore
$[st,z^w] = e$ because $T$ is a normal subgroup of $W$. By Lemma \ref{infiniteorder}
it follows that $Pc_S(st) = W$. As $st \in \langle s,t \rangle_W$ it follows that $S = \{ s,t \}$
and that $W$ is the infinite dihedral group. We conclude that $(W,S)$ is affine.

\smallskip
\noindent
Suppose now that there is $s \in S$ such that $z \neq szs \neq z^{-1}$.
We put $a := zszs$ and $b := z^{-1}szs$ and observe that $e \neq a \in Z$, $e \neq b \in Z$,
$sas = a$ and $sbs = b^{-1}$. Let $H = \langle a,b \rangle_W$. As $a,b \in Z$, the group $H$ is abelian
and as $a \neq e \neq b$ it follows by the argument above that they are both of infinite order.
We claim that $H$ is free abelian of rank 2. Indeed, let $k,m \in \mathbb{Z}$ be such that
$a^kb^m = e$. Then $e = sa^kb^ms = a^kb^{-m}$ which yields $b^{2m} = e$. As $b$ has infinite order,
we have $m=0$ which implies $a^k=e$ and finally $k=0$ since $a$ has infinite order.
Finally, as $e \neq a \in Z$, we have $Pc_S(a) = W$ by Lemma \ref{infiniteorder}
and therefore $Pc_S(H) = W$. Thus we may apply
Theorem \ref{Krammerthm} in order to conclude that $(W,S)$ is affine.
\end{proof}

\begin{theorem} \label{debt1} Let $(W,S)$ be an irreducible, non-spherical Coxeter system of finite rank.
Suppose that there are $x,y \in W$ such that $x \neq e \neq y$ and $[x,y^w] = e$ for all $w \in W$.
Then $(W,S)$ is affine.
\end{theorem}

\begin{proof} We put $N := \langle y^w : w \in W \rangle_W$ and observe that $N$ is a normal subgroup
of $W$ with $|N| \neq 1$. Moreover, $N \leq C_W(x)$. If $Z:= \langle x \rangle_W \cap N$ is non-trivial,
then $(W,S)$ is affine by Lemma \ref{centerofnormal}. Thus we are left with the case
where $\langle x \rangle_W \cap N$ is trivial. Now $x$ and $y$ have both infinite order by Lemma \ref{infiniteorder},
$[x,y]=e$ by assumption, and $\langle x \rangle_W \cap \langle y \rangle_W \leq \langle x \rangle_W \cap N$ is trivial.
We conclude that $H := \langle x,y \rangle_W$ is a free abelian subgroup of $W$. Furthermore, $Pc_S(x) = W$
by Lemma \ref{infiniteorder} and therefore $Pc_S(H) =W$. Now Theorem \ref{Krammerthm} yields that $(W,S)$
is affine.
\end{proof}

	\begin{definition}\label{def_domain} Let $G$ be a group. We say that $G$ is a {\em domain} if for every $x, y \in G$ with $x, y \neq e$ there exists $g \in G$ such that $[x, y^g] \neq e$.
\end{definition}

	\begin{proof}[Proof of Theorem~\ref{domain_conj}] This follows immediately from Lemma~\ref{debt1} and the fact that the direct product of two non-abelian groups is never a domain \cite{alg_geom_over_groups1, alg_geom_over_groups2}.
\end{proof}


\section{Elementary Substructures in RACGs}

	In this section we prove Theorem~\ref{el_substr}.
	
\medskip

	In previous sections we already used the notation which we are about to introduce, but we recall it for clarity, since it will appear in Lemma~\ref{bases_th}.
	
	\begin{notation} Given a group $G$ and $g, h \in G$ we denote $ghg^{-1}$ by $h^g$.	
\end{notation}

	\begin{remark} The formula $\varphi_{\Gamma}(\bar{x})$, for $\Gamma$ a finite graph, which we will introduce in Lemma~\ref{bases_th} plays a crucial role also in \cite{casal1}. In fact, as shown there, we have that if $W_\Theta$ is a right-angled Coxeter group of finite rank, then $W_\Theta \models \exists\bar{x} \varphi_{\Gamma}(\bar{x}) \; \text{ iff } \; \Gamma \cong \Theta.$
\end{remark}
		
	\begin{lemma}\label{bases_th} Let $(W_\Gamma, S)$ be a right-angled Coxeter system of finite rank and $S = \{ s_1, ..., s_n \}$. Let $\varphi_{\Gamma}(x_1, ..., x_n) = \varphi_{\Gamma}(\bar{x})$ be the first-order formula expressing:
	\begin{enumerate}[(a)]
	\item\label{item_a} for every $\ell \in [1, n]$, $x_\ell$ has order $2$ and $x_{\ell} \neq e$;
	\item\label{item_b} for every  $\ell \neq j \in [1, n]$, $x_{\ell} \neq x_j$ and $[x_\ell, x_j] = e$ if and only if $s_\ell E_\Gamma s_j$;
	\item\label{item_c} for every $\ell \in [1, n]$, $y_1, ..., y_{n} \in W_\Gamma$, and $k_1, ..., k_{\ell-1}, k_{\ell+1}, ..., k_{n} \in \{0, 1\}$:
	$$x_\ell^{y_{\ell}} \neq ((x_{1})^{k_1})^{y_{1}} \cdots ((x_{\ell-1})^{k_{\ell-1}})^{y_{\ell-1}} ((x_{\ell+1})^{k_{\ell+1}})^{y_{\ell+1}} \cdots ((x_{n})^{k_{n}})^{y_{n}}.$$
\end{enumerate}
	Then $W_\Gamma \models \varphi_\Gamma(g_1, ..., g_n)$ iff there is $\alpha \in F(\Gamma_S)$ (cf. Def.~\ref{F(Gamma)}) such that $\{ g_1, ..., g_n \}$ is a set of self-similar $T$-reflections of $W_\Gamma$ (cf. Def.~\ref{def_self_similar}), \mbox{where $T = \{ \alpha(s): s \in S \}$.}
\end{lemma}

	\begin{proof} The direction ``right-to-left'' is well-known, in fact conditions (a) and (b) and clear and condition (c) is also easily seen to be verified using Fact~\ref{abelianization_fact}.
Concerning the other direction, let $g_1, ..., g_n \in W_\Gamma$ and suppose that $W_\Gamma \models \varphi(g_1, ..., g_n)$. By condition (\ref{item_a}) of the definition of $\varphi_\Gamma(\bar{x})$ and Fact \ref{finite_order}, for every $\ell \in [1, n]$, we have:
	\begin{equation}\label{equation_F}
	g_{\ell} =  h_\ell a^\ell_1 \cdots a^\ell_{m(\ell)} (h_\ell)^{-1},
\end{equation}
with $A_\ell := \{ a^\ell_1, ..., a^\ell_{m(\ell)} \}$ inducing a non-empty clique of $\Gamma$. We claim that the map $\alpha$ determined by the assignment
$\hat{\alpha}: \{ s_{\ell} \} \mapsto A_\ell$
is in $F(\Gamma)$. First of all we claim that $\alpha$ is an automorphism of $V(\Gamma)$ (cf. Definition~\ref{F(Gamma)}). To see this it suffices to show that the set $ \{ A_\ell : \ell \in [1, n] \}$
is linearly independent in $V(\Gamma)$, and this is clear by condition~(\ref{item_c}) of the definition of $\varphi_\Gamma(\bar{x})$. In fact, suppose that this is not the case, then there exists $\ell \in [1, n]$ such that:
$$A_\ell = (A_1)^{k_1} \triangle \cdots \triangle (A_{\ell-1})^{k_{\ell-1}} \triangle (A_{\ell+1})^{k_{\ell+1}} \triangle \cdots \triangle (A_k)^{k_n},$$
with $k_1, ..., k_{\ell-1}, k_{\ell+1}, ..., k_{n} \in \{0, 1\}$ and $(A_i)^1 = A_i$ and $(A_i)^0 = \emptyset$, and so:
$$\prod A_\ell = (\prod A_1)^{k_1} \cdots (\prod A_{\ell-1})^{k_{\ell-1}} \cdots (\prod A_{\ell+1})^{k_{\ell+1}} \cdots (\prod A_k)^{k_n}.$$
Thus, letting $y_i = (h_i)^{-1}$ we have:
	$$g_\ell^{y_{\ell}} = ((x_{1})^{k_1})^{y_{1}} \cdots ((x_{\ell-1})^{k_{\ell-1}})^{y_{\ell-1}} ((x_{\ell+1})^{k_{\ell+1}})^{y_{\ell+1}} \cdots ((x_{n})^{k_{n}})^{y_{n}}.$$
contradicting (\ref{item_c}) of the definition of $\varphi_\Gamma(\bar{x})$ (cf.~(\ref{equation_F})).
Hence, in order to show that $\alpha \in F(\Gamma)$ we are only left with the verification that $\alpha$ sends cliques of $\Gamma$ to cliques of $\Gamma$, but this is clear by condition (\ref{item_b}) of the definition of $\varphi_\Gamma(\bar{x})$ and Fact~\ref{centralizer}. Hence, $\{ g_1, ..., g_n \}$ is a set of self-similar reflections of $(W_\Gamma, T)$, for $T = \{ \alpha(s): s \in S \}$.
\end{proof}

	\begin{definition}\label{reflection_property} Let $W$ be a right-angled Coxeter group of finite rank. We say that $W$ has the {\em self-similar reflection property} if for every Coxeter basis $S$ of $W$ and self-similar set of reflections $\hat{S}$ of $(W, S)$ (cf. Def.~\ref{def_self_similar}) we have that $\hat{S}$ generates $W$. On the other hand, we say that $W' \leq W$ is a {\em counterexample to the self-similar reflection property} if there exists a Coxeter basis $S$ of $W$ and a set $\hat{S}$ of self-similar reflections of $(W, S)$ such that $W' = \langle \hat{S} \rangle_W$ and $W'$ is a proper subgroup~of~$W$.
\end{definition}

	\begin{lemma}\label{lemma_elem_sbgp} Let $W = W_{\Gamma}$ be a right-angled Coxeter group of finite rank, and let $\varphi_\Gamma$ be the formula from Lemma~\ref{bases_th}.
%
	If $W' \lneq W$ is elementary in $W$ and $W'$ is a Coxeter group, then $W'$ is a counterexample to the self-similar reflection property.
\end{lemma}

	\begin{proof}
%
 Let $W = W_{\Gamma}$ be a right-angled Coxeter group of finite rank, and let $W' \lneq W$ be elementary in $W$. Suppose that $W'$ is a Coxeter group, then clearly $W'$ is right-angled (since any element in $W$ either has order $2$ or it has order $\infty$, cf. Fact~\ref{finite_order}). Since $W'$ is an elementary subgroup of $W$, then clearly $W'$ is elementary equivalent to $W$. First of all notice that $W'$ is of finite rank, since e.g. $W$ has finitely many conjugacy classes of involutions, and this is a first-order property. Thus, by the main result of \cite{casal1}, $W'$ is isomorphic to $W$. Let then $(W', T)$ be a right-angled Coxeter system of type $\Gamma$, with $T = \{ t_1, ..., t_{|\Gamma|} \}$ (recall Fact~\ref{rigidity}). Then $W' \models \varphi_{\Gamma}(t_1, ..., t_{|\Gamma|})$ and so $W \models \varphi_{\Gamma}(t_1, ..., t_{|\Gamma|})$. Thus, by Lemma~\ref{bases_th}, there exists a Coxeter basis $S$ of $W$ such that $T = \{t_1, ..., t_{|\Gamma|} \}$ is a set of self-similar reflections of $(W, S)$. Furthermore, clearly $\langle t_1, ..., t_{|\Gamma|} \rangle_W = W'$ and by hypothesis $W' \lneq W$.
%
\end{proof}

	\begin{lemma}\label{th_elem_Cox_sbgp} Let $W$ be a Coxeter group of finite rank. Then $W$ does {\em not} have proper elementary subgroups which are Coxeter group. 
\end{lemma}

	\begin{proof} Let $W' \leq W$ be elementary in $W$ and suppose that $W'$ is a Coxeter group and that $W' \lneq W$. Then, by
Lemma~\ref{lemma_elem_sbgp}, we have that $W'$ is a counterexample to the self-similar reflection property, i.e. there exists a Coxeter basis $S$ of $W$ and a set $\hat{S}$ of self-similar reflections of $(W, S)$ such that $W' = \langle \hat{S} \rangle_W$. By Proposition~\ref{muller_fact}, $W$ is isomorphic to $W'$ by the map $\alpha: \hat{s} \mapsto s$. Let $S = \{s_1, ..., s_n \}$, \mbox{$\hat{S} = \{ \hat{s}_1, ..., \hat{s}_n\}$, and let:}
$$\hat{s}_i = s_{i_1} \cdots s_{i_{k_i}} s_i s_{i_{k_i}} \cdots s_{i_1},$$
(recall that $\hat{s}_i \in s_i^W$) for $i \in [1, n]$. Now, clearly we have:
$$W \models \exists x_1, ..., x_n(\varphi_{\Gamma}(x_1, ..., x_n) \wedge \bigwedge_{i \in [1, n]} \hat{s}_i = x_{i_1} \cdots x_{i_{k_i}} x_i x_{i_{k_i}} \cdots x_{i_1}),$$
where $\Gamma$ is the graph specifying the type of $W$ and $\varphi_{\Gamma}(x_1, ..., x_n)$ is the formula from Lemma~\ref{bases_th}. Hence, being $W'$ elementary in $W$ and $\hat{S} \subseteq W'$ we have:
$$W' \models \exists x_1, ..., x_n(\varphi_{\Gamma}(x_1, ..., x_n) \wedge \bigwedge_{i \in [1, n]} \hat{s}_i = x_{i_1} \cdots x_{i_{k_i}} x_i x_{i_{k_i}} \cdots x_{i_1}).$$
But then, via the isomorphism $\alpha: W' \cong W$ such that $\hat{s} \mapsto s$, we have:
\begin{equation}\label{equstar} W \models \exists x_1, ..., x_n(\varphi_{\Gamma}(x_1, ..., x_n) \wedge \bigwedge_{i \in [1, n]} s_i = x_{i_1} \cdots x_{i_{k_i}} x_i x_{i_{k_i}} \cdots x_{i_1}). \tag{$\star$}
\end{equation}
Let $b_1, ..., b_n \in W$ be a witness of (\ref{equstar}). Then, by Lemma~\ref{bases_th}, there exists a Coxeter basis $T$ of $W$ such that $B :=\{ b_1, ..., b_n \}$ is a set of self-similar reflections of $(W, T)$. On the other hand, by the second conjunct of the formula in (\ref{equstar}), we have that:
$$\langle b_i : i \in [1, n] \rangle_W = W,$$
since $S \subseteq \langle b_i : i \in [1, n] \rangle_W$ and $S$ generates $W$. Thus, by Proposition~\ref{muller_fact}, we have that $B$ is a basis of $W$. Hence, we have:
$$\beta: s_i \mapsto b_i \in Aut(W),$$
by Fact~\ref{rigidity} and the fact that $B$ is a basis of $W$. Furthermore:
$$ \{b_{i_1} \cdots b_{i_{k_i}} b_i b_{i_{k_i}} \cdots b_{i_1} : i \in [1, n]\} = \{s_1, ..., s_n \} = S$$
is a basis of $W$, and so:
$$\gamma: b_i \mapsto b_{i_1} \cdots b_{i_{k_i}} b_i b_{i_{k_i}} \cdots b_{i_1} \in Aut(W).$$
Hence, we have:
$$\begin{array}{rcl}
(\beta^{-1} \circ \gamma \circ \beta)(s_i) & = & (\beta^{-1} \circ \gamma)(b_i)\\
 & = & \beta^{-1}(b_{i_1} \cdots b_{i_{k_i}} b_i b_{i_{k_i}} \cdots b_{i_1})\\
  & = & \beta^{-1}(b_{i_1}) \cdots \beta^{-1}(b_{i_{k_i}}) \beta^{-1}(b_i) \beta^{-1}(b_{i_{k_i}}) \cdots \beta^{-1}(b_{i_1}) \\
 & = & s_{i_1} \cdots s_{i_{k_i}} s_i s_{i_{k_i}} \cdots s_{i_1},
\end{array}$$
and so the map $\alpha^{-1}: s_i \mapsto \hat{s}_i = s_{i_1} \cdots s_{i_{k_i}} s_i s_{i_{k_i}} \cdots s_{i_1}$ is an automorphism of $W$, contradicting the fact that $W' = \langle \hat{S} \rangle_W$ is a proper subgroup of $W$.
\end{proof}

	\begin{lemma}\label{refl_th} Let $\Gamma$ be a graph of arbitrary cardinality with the star-property (cf. Definition \ref{star_prop}). Let $\psi(x)$ be the first-order formula expressing:
	\begin{enumerate}[(1)]
	\item $x$ has order $2$ and $x \neq e$;
	\item there is no $y$ of order $2$ such that $e \neq y \neq x$ and for every $z$ of order $2$ we have:
	$$[z, x] = e \text{ implies } [z, y] = e.$$
\end{enumerate}
Then the following are equivalent for $a \in W_\Gamma = W$:
\begin{enumerate}[(1)]
	\item $W \models \psi(a)$;
	\item $a \in S^W$, for some (equivalently, every) Coxeter basis $S$ of $W$.
\end{enumerate}
\end{lemma}

\begin{proof} Suppose that $\Gamma$ has the star property and let $W = W_{\Gamma}$. By Fact~\ref{reflection_preserving}, $S^W$ does not depend on a choice of the Coxeter basis $S$ of $W$, let then $R(W) := S^W$. First of all, we prove that for $g \in R(W)$ we have that $W \models \psi(g)$. Notice that it suffices to show this for $s \in S$, since conjugation is an automorphism. Let then $s \in S$ and $b \in W$ of order $2$ with $e \neq b \neq g$. We want to find $c \in C_{W}(s, 2) - C_{W}(b, 2)$, where for $h \in W$ we let $C_{W}(h, 2) = C_{W}(h) \cap \{ g \in W : g^2 = e \}$, and we recall that $C_{W}(h)$ denotes the centralizer of $h$ in $W$. By Fact \ref{finite_order}, we have that $b = w = s_1 \cdots s_m a s_m \cdots s_1$, with $w$ reduced and $sp(a)$ inducing a non-empty clique of $\Gamma$. We make a case distinction:
\newline {\em Case 1}. $\ell(a) = 1$ and $a = t \in S$ with $t \neq s$.
\newline Let $r \in N^*(s) - N^*(t)$. Then $r = c$ is as wanted.
\newline {\em Case 2}. $\ell(a) = 1$ and $a = s$.
\newline In this case necessarily $m \geq 1$, since we are assuming that $g \neq b$. Let $r \in N^*(s) - N^*(s_m)$. Then $r = c$ is as wanted.
\newline {\em Case 3}. $\ell(a) > 1$.
\newline Let $t \in sp(a) - \{ s \}$ and $r \in N^*(s) - N^*(t)$. Then $r = c$ is as wanted.
\newline We now prove that if $W \models \psi(g)$, then $g \in R(W)$. Now, since $g$ is of order $2$, by Fact \ref{finite_order}, we have that $g = w = s_1 \cdots s_m a s_m \cdots s_1$, with $w$ reduced and $sp(a)$ inducing a non-empty clique of $\Gamma$. For the sake of contradiction, suppose that $a = w' = t_1 \cdots t_k$ with $k = \ell(a) \geq 2$. Then, for every $\ell \in [1, k]$, we have that:
	$$C_{W}(g, 2) \subsetneq C_{W}(t_{\ell}, 2),$$
where the inclusion $\subseteq$ is by Fact \ref{centralizer}, and the fact that the inclusion is proper is by the star-property, and so $W \not\models \psi(g)$, a contradiction. So \mbox{$\ell(a) = 1$ and $g \in R(W)$.}
\end{proof}

\begin{theorem}\label{interpretability_th} Let $\Gamma$ be a right-angled graph (finite or infinite) and $W_\Gamma = W$ the corresponding right-angled Coxeter group. Then the following are equivalent:	
	\begin{enumerate}[(1)]
	\item $\Gamma$ has the star-property (cf.~Definition~\ref{star_prop});
	\item the set of reflections $S^{W}$ of the Coxeter system $(W, S)$ is invariant under change of basis $S$ of $W$;
	\item the set of reflections $S^{W}$ of the Coxeter system $(W, S)$ is invariant under change of basis $S$ of $W$ and it is first-order definable in $W$ without parameters.
	\end{enumerate}
Furthermore, if $\Gamma$ has the star-property, then the graph $\Gamma$ is interpretable in $W_{\Gamma}$.
\end{theorem}

	\begin{proof} The equivalence of (1) and (2) is by Fact~\ref{reflection_preserving}, while the other equivalence is by Lemma~\ref{refl_th}. The ``furthermore'' also follows easily from Lemma~\ref{refl_th}.
\end{proof}

	\begin{corollary}\label{cor_el_sbg_star} Let $W_\Gamma$ be a right-angled Coxeter group of finite rank, and suppose that $\Gamma$ has the star-property. Then $W_\Gamma$ does not have proper elementary subgroups.
\end{corollary}

	\begin{proof} Suppose that $\Gamma$ has the star property and let $W = W_{\Gamma}$. By Fact~\ref{reflection_preserving}, $S^W$ does not depend on a choice of the Coxeter basis $S$ of $W$, let then $R(W) := S^W$. Let $W'$ be an elementary subgroup of $W$. By Lemma~\ref{th_elem_Cox_sbgp} and \cite{deo, dyer} it suffices to show that $W$ is a reflection subgroup of $W$ (since then we have that $W'$ is a Coxeter group and we can indeed apply Lemma~\ref{th_elem_Cox_sbgp}), i.e. $W' = \langle W' \cap R(W) \rangle_W$. Let $\psi(x)$ the formula defining $R(W)$ in $W$ (cf. Theorem~\ref{interpretability_th}). Let $a \in W'$, then:
	$$W \models \exists x_1, ..., x_n(\bigwedge_{i \in [1, n]} \psi(x_i) \wedge a = x_1 \cdots x_n),$$
and thus:
	$$W' \models \exists x_1, ..., x_n(\bigwedge_{i \in [1, n]} \psi(x_i) \wedge a = x_1 \cdots x_n).$$
Hence, since $\psi(W') \subseteq \psi(W)$ being $W'$ elementary in $W$, we can find $t_1, ..., t_n \in W' \cap R(W)$ such that $a = t_1 \cdots t_n$, and so $a \in \langle W' \cap R(W) \rangle_W$, as wanted.
\end{proof}

	\begin{proof}[Proof of Theorem~\ref{el_substr}] Immediate by Lemmas~\ref{th_elem_Cox_sbgp}~and~\ref{refl_th}, and Corollary~\ref{cor_el_sbg_star}.
\end{proof}



\section{Prime Models in RACGs}\label{sec_prime_models}
	
		In this section we prove Theorem~\ref{th_prime_models}.

\subsection{Prime Models and $Sim(W, S)$}

	\begin{proposition}\label{type_def_basis} Let $W$ be a right-angled Coxeter group of finite rank. Then the $Aut(W)$-orbit of any Coxeter basis is type-definable in $W$ without parameters.
\end{proposition}

	\begin{proof} Let $\Gamma$ be the type of $W$, $\varphi_{\Gamma}$ be as in the proof of Lemma~\ref{bases_th}, and $n = |\Gamma|$. Let $X = \varphi_{\Gamma}(M) = \{\bar{a} \in W^n : W \models \varphi_{\Gamma}(\bar{a}) \}$ and $X_* = \{\bar{a} \in X : \langle \bar{a} \rangle_W \neq W\}$. Now, by Lemma~\ref{bases_th}, for every $\bar{a} \in X$, there exists a basis $T_{\bar{a}}$ such that $\{ a_1, ..., a_n \}$ is a set of self-similar reflections of $(W, T_{\bar{a}})$. For every $\bar{a} \in X_*$, fix one such basis $T_{\bar{a}}$ and an enumeration $\leq_{\bar{a}} = \{t_{(\bar{a}, 1)}, ..., t_{(\bar{a}, n)} \}$ of $T_{\bar{a}}$. Then for every $\bar{a} \in X_*$ and for $i \in [1, n]$ we have a $T_{\bar{a}}$-normal form:
	$$a_i = t_{(\bar{a}, i_1)} \cdots t_{(\bar{a}, i_{k_i})} t_{(\bar{a}, i)} t_{(\bar{a}, i_{k_i})} \cdots t_{(\bar{a}, i_1)}.$$
Thus, for every $\bar{a} \in X_*$, let:
$$\chi_{(\bar{a}, T_{\bar{a}}, \leq_{\bar{a}})}(\bar{x}, \bar{y}) = \bigwedge_{i \in [1, n]} x_i = y_{i_1} \cdots y_{i_{k_i}} y_i y_{i_{k_i}} \cdots y_{i_1},$$
$$\theta_{(\bar{a}, T_{\bar{a}}, \leq_{\bar{a}})}(\bar{x}) = \neg \exists y_1, ..., y_n (\varphi_{\Gamma}(\bar{y}) \wedge \chi_{(\bar{a}, T_{\bar{a}}, \leq_{\bar{a}})}(\bar{x}, \bar{y})).$$
Let then:
$$p_{\Gamma}(\bar{x}) = \{ \varphi_{\Gamma}(\bar{x}) \} \cup \{ \theta_{(\bar{a}, T_{\bar{a}}, \leq_{\bar{a}})}(\bar{x}) : \bar{a} \in X_* \}.$$
We claim that $\bar{a} = (a_1, ..., a_n) \models p_{\Gamma}(\bar{x})$ if and only if $\{ a_1, ..., a_n \}$ is a basis of $W$.
Concerning the implication ``left-to-right'', suppose that $\bar{b} = (b_1, ..., b_n) \models p_{\Gamma}(\bar{x})$, then $\bar{b} \in X$, and so it suffices to show that $b \notin X_*$, since then by Proposition~\ref{muller_fact} we have that $\{ b_1, ..., b_n \}$ is a basis of $W$. For the sake of contradiction, suppose that $\bar{b} \in X_*$, then the basis $T_{\bar{b}} = \{t_{(\bar{b}, 1)}, ..., t_{(\bar{b}, n)} \}$ is such that:
	$$b_i = t_{(\bar{b}, i_1)} \cdots t_{(\bar{b}, i_{k_i})} t_{(\bar{b}, i)} t_{(\bar{b}, i_{k_i})} \cdots t_{(\bar{b}, i_1)},$$	
and so $W \models \exists y_1, ..., y_n (\varphi_{\Gamma}(\bar{y}) \wedge \chi_{(\bar{b}, T_{\bar{b}}, \leq_{\bar{b}})}(\bar{b}, \bar{y}))$, contradicting the fact that:
$$ \neg \exists y_1, ..., y_n (\varphi_{\Gamma}(\bar{y}) \wedge \chi_{(\bar{b}, T_{\bar{b}})}(\bar{b}, \bar{y})) \in p_{\Gamma}(\bar{x}).$$
Concerning the implication ``right-to-left'', let $(s_1, ..., s_n) = \bar{s}$ be a basis of $W$, we want to show that $\bar{s} \models p_{\Gamma}(\bar{x})$. Clearly, $W \models \varphi_{\Gamma}(\bar{s})$. For the sake of contradiction, suppose that for some $\bar{a} \in X_*$ we have:
$$W \models \exists y_1, ..., y_n (\varphi_{\Gamma}(\bar{y}) \wedge \chi_{(\bar{a}, T_{\bar{a}}, \leq_{\bar{a}})}(\bar{s}, \bar{y})).$$
Then there exists $b_1, ..., b_n \in W$ such that:
$$W \models \varphi_{\Gamma}(b_1, ..., b_n) \wedge \bigwedge_{i \in [1, n]} s_i = b_{i_1} \cdots b_{i_{k_i}} b_i b_{i_{k_i}} \cdots b_{i_1},$$
where $s_i = y_{i_1} \cdots y_{i_{k_i}} y_i y_{i_{k_i}} \cdots y_{i_1}$ is the formula $\chi_{(\bar{a}, T_{\bar{a}}, \leq_{\bar{a}})}(\bar{s}, \bar{y})$. But then, arguing as in the proof of Theorem~\ref{th_elem_Cox_sbgp} we see that $\{b_1, ..., b_n \}$ is a basis of $W$, and that:
$$ \beta: t_i \mapsto b_i \in Aut(W), \;\; \gamma : b_i \mapsto b_{i_1} \cdots b_{i_{k_i}} b_i b_{i_{k_i}} \cdots b_{i_1} \in Aut(W),$$
where $T_{\bar{a}} = \{t_{(\bar{a}, 1)}, ..., t_{(\bar{a}, n)} \} = \{t_1, ..., t_n \}$. Thus, exactly as in in the proof of Theorem~\ref{th_elem_Cox_sbgp}, we see that:
$$(\beta^{-1} \circ \gamma \circ \beta)(t_i) = t_{i_1} \cdots t_{i_{k_i}} t_i t_{i_{k_i}} \cdots t_{i_1},$$
and so:
$$t_i \mapsto t_{i_1} \cdots t_{i_{k_i}} t_i t_{i_{k_i}} \cdots t_{i_1} = a_i \in Aut(W)$$
contradicting the fact $\bar{a} \in X_{*}$.
\end{proof}

	We invite the reader to recall the definition of $Sim^*(W, S)$ from Definition~\ref{special_monoid}.

\begin{lemma}\label{lemma_no_orbit} In the context of the proof of Proposition~\ref{type_def_basis}, in the definition:
	 $$ p_{\Gamma}(\bar{x}) = \{ \varphi_{\Gamma}(\bar{x}) \} \cup \{ \theta_{(\bar{a}, T_{\bar{a}}, \leq_{\bar{a}})}(\bar{x}) : \bar{a} \in X_* \}$$
we can assume that for every $\bar{a} \in X_*$ we have that:
\begin{enumerate}[(i)]
	\item $\bar{a}$ is a set of self-similar reflections of $(W, S)$ for a fixed basis $S$ of $W$;
	\item $\bar{a} = (a_1, ..., a_n) = (\alpha(s_1), ..., \alpha(s_n))$, for some $\alpha \in Sim^*(W, S)$;
	\item $\theta_{(\bar{a}, T_{\bar{a}}, \leq_{\bar{a}})}(\bar{x}) = \theta_{(\bar{a}, S, \leq)}(\bar{x})$, for a fixed enumeration $
\leq \, = \{s_1, ..., s_n \}$~of~$S$.
\end{enumerate}
Consequently, from now on we denote $X_*$ simply as $X_S$, and we let $X_S$ be:
	$$\{\bar{a} \in \varphi_{\Gamma}(W) : \bar{a} = (a_1, ..., a_n) = (\alpha(s_1), ..., \alpha(s_n)), \text{ for } \alpha \in Sim^*(W, S) \}.$$
Also, we let:
	 $$p_{\Gamma}(\bar{x}) = \{ \varphi_{\Gamma}(\bar{x}) \} \cup \{ \theta_{(\bar{a}, S, \leq)}(\bar{x}) : \bar{a} \in X_S \}.$$
Finally notice that in this notation we have that:
$$\theta_{(\bar{a}, S, \leq)}(\bar{x}) = \neg \exists y_1, ..., y_n (\varphi_{\Gamma}(\bar{y}) \wedge \chi_{(\bar{a}, S, \leq)}(\bar{x}, \bar{y})),$$
where:
$$\chi_{(\bar{a}, S, \leq)}(\bar{x}, \bar{y})) = \bigwedge_{i \in [1, n]} x_i = y_{i_1} \cdots y_{i_{k_i}} y_i y_{i_{k_i}} \cdots y_{i_1}.$$
\end{lemma}

	\begin{proof}
This is clear from the proof of Proposition~\ref{type_def_basis}. Notice in fact that if $S$ is a basis of $W$ and $\bar{a} \in X_*$, then, by Fact~\ref{rigidity}, we can find $\beta \in Aut(W)$ such that $\beta(S)= T_{\bar{a}}$ and thus $\beta^{-1}(\bar{a}) = \bar{b} \in X_*$. Furthermore, since the proof of Proposition~\ref{type_def_basis} did not depend on the choice function $\bar{a} \mapsto \, \leq_{\bar{a}} \, = \{t_{(\bar{a}, 1)}, ..., t_{(\bar{a}, n)} \} = \{t_1, ..., t_n \}$, we can always choose $\leq_{\bar{a}} \, = \{ \beta(s_1), ..., \beta(s_n) \}$, and so we have that $b_i \in s_i^W$, since $a_i \in t_i^W$. Thus, $\alpha: s_i \mapsto b_i \in Sim^*(W, S)$ and we have:
	$$W \models \neg \theta_{(\bar{b}, S, \leq)}(\bar{a}).$$
Hence, $\bar{a} \models p_{\Gamma}(\bar{x})$ if and only if $\bar{a} \models \{ \varphi_{\Gamma}(\bar{x}) \} \cup \{ \theta_{(\bar{a}, S, \leq)}(\bar{x}) : \bar{a} \in X_S \}$.
\end{proof}

	\begin{notation}\label{Spe_notation} In the context of Lemma~\ref{lemma_no_orbit}, for $\bar{a} \in X_S$, let $f_{\bar{a}}: s_i \mapsto a_i$, for $i \in [1, n]$. Then $f_{\bar{a}} \in Sim^*(W, S)$, and the map $\bar{a} \mapsto f_{\bar{a}}$ is a bijection of $X_S$ onto $Sim^*(W, S)$. Hence, letting $\bar{s} = \{s_1, ..., s_n\}$, for $f \in Sim^*(W, S)$ we let:
	$$\chi_{(f(\bar{s}), S, \leq)}(\bar{x}) = \chi_{(f, \bar{s})}(\bar{x}),$$
	$$\theta_{(f(\bar{s}), S, \leq)}(\bar{x}) = \theta_{(f, \bar{s})}(\bar{x}).$$
Also, in this notation, we let:
	$$p_{\Gamma}(\bar{x}) = \{ \varphi_{\Gamma}(\bar{x}) \} \cup \{ \theta_{(f, \bar{s})}(\bar{x}) : f \in Sim^*(W, S) \}.$$
Finally, given $f \in Sim^*(W, S)$ we denote by $\bar{a}_f$ the associated element of $X_S$.
\end{notation}

We invite the reader to recall the definition of $Sim(W, S)$ from Definition~\ref{special_monoid}.

	\begin{lemma}\label{connection_remark} In the context of Notation~\ref{Spe_notation}, let $f_{\bar{a}}, f_{\bar{b}} \in Sim^*(W, S)$. Suppose that there exists $g \in Sim_\star(W, S)$ such that $f_{\bar{a}} = g \circ f_{\bar{b}}$, then:
	$$W \models \exists y_1, ..., y_n (\varphi_{\Gamma}(\bar{y}) \wedge \chi_{(f_{\bar{b}}, \bar{s})}(\bar{a}, \bar{y})).$$
\end{lemma}

	\begin{proof} 
	By definition we have:
$$W \models \chi_{(f_{\bar{b}}, \bar{s})}(\bar{b}, \bar{s}).$$
And so, since $g$ is a monomorphism we have:
$$W \models \chi_{(f_{\bar{b}}, \bar{s})}(g(\bar{b}), g(\bar{s})).$$
Furthermore, we have:
\begin{enumerate}[(i)]
	\item $f_{\bar{b}}(\bar{s}) = \bar{b}$ and $f_{\bar{a}}(\bar{s}) = \bar{a}$;
	\item $g(\bar{b}) = g(f_{\bar{b}}(\bar{s})) = f_{\bar{a}}(\bar{s}) = \bar{a}$;
	\item $W \models \varphi_{\Gamma}(g(\bar{s}))$, since $g \in Sim(W, S)$ (cf. Lemma~\ref{bases_th}).
\end{enumerate}
Hence:
$$W \models \varphi_{\Gamma}(g(\bar{s})) \wedge \chi_{(f_{\bar{b}}, \bar{s})}(\bar{a}, g(\bar{s})).$$
\end{proof}

	\begin{definition} Let $f \in Sim^*(W, S)$, we say that $f$ is indecomposable if there are no $g \in Sim^*(W, S)$ and $h \in Sim^*(W)$ such that $f = g \circ h$.
\end{definition}

	\begin{lemma}\label{remark_proper_subgroup} Let $f_{\bar{a}} \in Sim^*(W, S)$ and suppose that $f_{\bar{a}}$ is not indecomposable, and let $f_{\bar{b}}, f_{\bar{c}} \in Sim^*(W, S)$ be such that $f_{\bar{a}} = f_{\bar{c}} \circ f_{\bar{b}}$, then $\langle f_{\bar{a}}(\bar{s}) \rangle_W \lneq \langle f_{\bar{c}}(\bar{s}) \rangle_W \lneq W$.
\end{lemma}

	\begin{proof} The fact that $\langle f_{\bar{c}}(\bar{s}) \rangle_W \lneq W$ is clear, since by hypothesis $f_{\bar{c}} \in Sim^*(W, S)$. Now, by the same argument used in the proof of Lemma~\ref{connection_remark} we have that:
$$W \models \chi_{(f_{\bar{b}}, \bar{s})}(f_{\bar{c}}(\bar{b}), f_{\bar{c}}(\bar{s})).$$
Thus, $\langle f_{\bar{a}}(\bar{s}) \rangle_W \leq \langle f_{\bar{c}}(\bar{s}) \rangle_W$, since every element of $f_{\bar{c}}(\bar{b}) = f_{\bar{c}}(f_{\bar{b}}(\bar{s})) = f_{\bar{a}}(\bar{s})$ can written as a product of elements of $f_{\bar{c}}(\bar{s})$ (cf. the explicit definition of $\chi_{(f_{\bar{b}}, \bar{s})}(\bar{x}, \bar{y})$ from Notation~\ref{Spe_notation}).
Furthermore, the map $\phi: \langle f_{\bar{c}}(\bar{s}) \rangle_W \rightarrow \langle f_{\bar{c}}(\bar{s}) \rangle_W$
such that
$f_{\bar{c}}(\bar{s}) \mapsto f_{\bar{c}}(\bar{b}) = f_{\bar{c}}(f_{\bar{b}}(\bar{s})) = f_{\bar{a}}(\bar{s})$ is not surjective, since $\langle f_{\bar{c}}(\bar{s}) \rangle_W \cong W$ (say via the map $\gamma$), in fact if it were we could find group words $w_1(\bar{x}), ..., w_n(\bar{x})$ such that $f_{\bar{c}}(s_i) = w_i(f_{\bar{c}}(\bar{b}))$, for every $i \in [1, n]$, and so via the isomorphism $\gamma$ we would have that the same is true for $\bar{s}$ and $\bar{b}$, contradicting that $f_{\bar{b}} \in Sim^*(W, S)$, i.e. contradicting that $f_{\bar{b}}$ is {\em not} surjective. Hence, $\langle f_{\bar{a}}(\bar{s}) \rangle_W \lneq \langle f_{\bar{c}}(\bar{s}) \rangle_W \lneq W$.
\end{proof}

	\begin{lemma}\label{theorem_prime_model} Let $B$ be a generating set for the monoid $Sim(W, S)$ and let $Y_B = \{f \in B : f \in Sim^*(W, S) \}$. Then the type $p_{\Gamma}(\bar{x})$ is isolated by the type:
	$$p^B_{\Gamma}(\bar{x}) = \{ \varphi_{\Gamma}(\bar{x}) \} \cup \{\theta_{(f, \bar{s})}(\bar{x}) : f \in Y_B \}.$$
Further, if $Z \subseteq Y_B$, $f_* \notin \langle Z \cup Spe(W) \rangle_{Sim(W, S)}$ and $f_* \in Sim^*(W, S)$ is indecomposable, then the type $\{ \varphi_{\Gamma}(\bar{x}) \} \cup \{\theta_{(f, \bar{s})}(\bar{x}) : f \in Z \} \text{ does not imply } \{ \theta_{(f_*, \bar{s})}(\bar{x}) \}$, that is there is $\bar{a} \in W^{< \omega}$ s.t. $W \models \{ \varphi_{\Gamma}(\bar{a}) \} \cup \{\theta_{(f, \bar{s})}(\bar{a}) : f \in Z \}$ and $W \not\models \theta_{(f_*, \bar{s})}(\bar{a})$.
\end{lemma}

	\begin{proof} Concerning the first claim, let $\bar{a} \in X_S$, we want to show that $\bar{a} \not \models p^B_{\Gamma}(\bar{x})$. If $f_{\bar{a}} \in Y_B$, then this is clear (cf. the proof of Proposition~\ref{type_def_basis}). Suppose then that $f_{\bar{a}} \in X_S - Y_B$, then there are $f_{\bar{b}} \in  Y_B$,  $h \in Spe(W)$ and $g \in Sim(W, S)$ such that:
	 $$f_{\bar{a}} = g \circ f_{\bar{b}} \circ h,$$
in fact since $Sim(W, S)$ is generated by $B$ we have that $f_{\bar{a}} = \alpha_1 \circ \cdots \circ \alpha_k$ with all the $\alpha_i \in B$ and such that at least one of the $\alpha_i \in Sim^*(W, S)$ (since if they were all automorphisms, then also $f_{\bar{a}}$ would be an automorphism, contrary to our assumption that $f_{\bar{a}} \in Sim^*(W, S))$, and so letting $f_{\bar{b}}$ to be the largest $i \in [1, k]$ such that $\alpha_i \in Sim^*(W, S)$ we are done (notice that if $i = k$ we can take $h = id_W$).

\smallskip
\noindent Now, simply unravelling notations, we have:
$$g \circ f_{\bar{b}} \circ h(\bar{s}) = f_{\bar{a}}(\bar{s}) = \bar{a},$$
and so, letting $\bar{a}' = h^{-1} (\bar{a})$ (recall that $h$ is an automorphism), we have that:
$$g \circ f_{\bar{b}} (\bar{s}) = \bar{a}' = f_{\bar{a}'}(\bar{s}),$$
where, clearly $f_{\bar{a}'} \in Sim^*(W, S)$, since $f_{\bar{b}} \in Sim^*(W, S)$. Hence, by Lemma~\ref{connection_remark}:
	$$W \models \exists y_1, ..., y_n (\varphi_{\Gamma}(\bar{y}) \wedge \chi_{(f_{\bar{b}}, \bar{s})}(\bar{a}', \bar{y})).$$
And since $h \in Aut(W)$ and $h^{-1}(\bar{a}') = \bar{a}$, clearly we have that:
	$$W \models \exists y_1, ..., y_n (\varphi_{\Gamma}(\bar{y}) \wedge \chi_{(f_{\bar{b}}, \bar{s})}(\bar{a}, \bar{y})).$$
Hence, $\bar{a} \not \models p^B_{\Gamma}(\bar{x})$, as $\theta_{(f_{\bar{b}}, \bar{s})}(\bar{x}) \in p^B_{\Gamma}(\bar{x})$, since $f_{\bar{b}} \in Y_B = \{f \in B : f \in Sim^*(W, S) \}$.

\smallskip
\noindent	Concerning the ``furthermore claim'', let $Z \subseteq Y_B$ and let:
$$f_* = f_{\bar{a}} \notin \langle Z \cup Spe(W) \rangle_{Sim(W, S)}$$
be such that $f_* \in Sim^*(W, S)$ is indecomposable. Clearly $W \models \neg \theta_{(f_{\bar{a}}, \bar{s})}(\bar{a})$. We claim that $\bar{a} \models \{ \varphi_{\Gamma}(\bar{x}) \} \cup \{\theta_{(f, \bar{s})}(\bar{x}) : f \in Z \}$, and so $\{ \varphi_{\Gamma}(\bar{x}) \} \cup \{\theta_{(f, \bar{s})}(\bar{x}) : f \in Z \}$ does not imply $\theta_{(f_{\bar{a}}, \bar{s})}(\bar{x})$.
Suppose that this is not true, and let \mbox{$f_{\bar{b}} \in Z$ be such that:}
	$$W \models \exists \bar{y} (\varphi_{\Gamma}(\bar{y}) \wedge \chi_{(f_{\bar{b}}, \bar{s})} (\bar{a}, \bar{y})).$$
	Then we can find $\bar{t} \in W^{<\omega}$ such that:
	\begin{equation}\label{eq_ai_ti} W \models \varphi_{\Gamma}(\bar{t}) \wedge \chi_{(f_{\bar{b}}, \bar{s})} (\bar{a}, \bar{t}). \tag{$*$}
\end{equation}
Let now, $g: s_i \mapsto t_i$, for $i \in [1, n]$. We claim that $t_i \in s_i^W$, i.e. that $g \in Sim(W, S)$. Let $i \in [1, n]$, then $a_i \in s_i^W$ by hypothesis, since $f_{\bar{a}} \in Sim^*(W, S)$. Furthermore, by (\ref{eq_ai_ti}) above we have that $a_i \in t_i^W$. Let then $w, u \in W$ be such that:
$$ws_iw^{-1} = a_i = ut_iu^{-1}.$$
Then we have:
$$t_i = u^{-1} ws_iw^{-1} u = u^{-1} ws_i(u^{-1} w)^{-1} \in s_i^W.$$
We now claim that $f_{\bar{a}} = g \circ f_{\bar{b}}$. To this extent, let:
$$\chi_{(f_{\bar{b}}, \bar{s})}(\bar{x}, \bar{y}) = \bigwedge_{i \in [1, n]} x_i = y_{i_1} \cdots y_{i_{k_i}} y_i y_{i_{k_i}} \cdots y_{i_1}.$$
Then we have (where in the last equation we use  crucially (\ref{eq_ai_ti})):
\[ \begin{array}{rcl}
	g \circ f_{\bar{b}}(s_i) & = & g(s_{i_1} \cdots s_{i_{k_i}} s_i s_{i_{k_i}} \cdots s_{i_1}) \\
	& = & g(s_{i_1}) \cdots g(s_{i_{k_i}}) g(s_i) g(s_{i_{k_i}}) \cdots g(s_{i_1})) \\
	& = & t_{i_1} \cdots t_{i_{k_i}} t_i t_{i_{k_i}} \cdots t_{i_1}\\
	& = & a_i.\\
\end{array}	\]
\newline \underline{\em Case 1}. $g \in Spe(W)$.
\newline In this case $f_{\bar{a}} = g \circ f_{\bar{b}} \in \langle Z \cup Spe(W) \rangle_{Sim(W, S)}$, since $f_{\bar{b}} \in Sim^*(W, S)$.
\newline \underline{\em Case 2}. $g \in Sim^*(W, S)$.
\newline In this case $f_{\bar{a}} = g \circ f_{\bar{b}}$, with $g, f_{\bar{b}} \in Sim^*(W, S)$ and so $f_*$ is not indecomposable.
\end{proof}

	\begin{proposition}\label{product_of_indecomposables} 
	For every $f \in Sim^*(W, S)$ there exist $f_1, ..., f_n \in Sim^*(W, S)$ such that $f = f_1 \circ \cdots \circ f_n$ and, for every $i \in [1, n]$, $f_i$ is indecomposable.
\end{proposition}

	\begin{proof} This is an immediate consequence of Coroolary~\ref{product_of_indecomposables_pre}.
\end{proof}

	\begin{corollary}\label{indecomposable_lemma} 
If the monoid $Sim(W, S)$ is {\em not} finitely generated, then for every finite $Z \subseteq Sim^*(W, S)$ there exists an indecomposable $f_* \notin \langle Z \cup Spe(W) \rangle_{Sim(W, S)}$.
\end{corollary}

	\begin{proof} This is a consequence of Proposition~\ref{product_of_indecomposables}, since we have that:
	$$Sim(W, S) = \langle Sim^{*}(W, S) \cup Spe(W) \rangle_{Sim(W, S)}$$
and $Spe(W)$ is generated by finitely many involutory automorphisms (cf. Fact~\ref{partial_conj_fact}).
\end{proof}


\subsection{Generators of $Sim(W, S)$} In this section we prove the ``furthermore'' part of Theorem~\ref{th_prime_models}, i.e. that if $W$ is a universal Coxeter group of finite rank at least two and $S$ is a basis of $W$, then the monoid $Sim(W, S)$ is {\em not} finitely generated.

\medskip


Let $G$ be a group and let $\alpha$ be an endomorphism of $G$. Then $\alpha(e_G) = e_G$ and $\alpha(x^{-1}) = \alpha(x)^{-1}$
for each $x \in G$. It follows that $\alpha([G,G]) \subseteq [G,G]$ for all $\alpha \in End(G)$ and therefore
each such $\alpha$ induces an endomorphism $\overline{\alpha}$ of $\overline{G} := G/[G,G]$ (namely $g[G,G] \mapsto \alpha(g)/[G,G]$).
Thus, we have a natural homomorphism of monoids $\pi$ from $End(G)$ to $End(\overline{G})$ (namely $\alpha \mapsto \overline{\alpha}$).

\medskip
\noindent
Let $L$ be the free abelian group of rank $n$ where $1 \leq n < \omega$. Then $det$ is a homomorphism of monoids
from $End(L)$ to the monoid $\mathbb{Z}$ with multiplication. Furthermore, if $M$ is a submonoid of $End(L)$ such that
$det(M)$ contains infinitely many prime numbers, then $M$ is not finitely generated.

\medskip
\noindent
Let $G$ be a group and suppose that $\overline{G} := G/[G,G]$ is a free abelian group of rank $n$ where $1 \leq n < \omega$
and suppose that $M$ is a submonoid of $End(G)$. If $det \circ \pi(M)$ contains infinitely many prime numbers, then $M$ is not
finitely generated.


	\begin{lemma}\label{the_hammer} Let $W$ be a universal Coxeter group of finite rank at least two and $S$ a basis of $W$, then the monoid $Sim(W, S)$ is {\em not} finitely generated
\end{lemma}

	\begin{proof} Let $(W,S)$ be the universal Coxeter system of rank $n+1$ where $1 \leq n < \omega$ and let $S = \{ s_0,s_1,\ldots,s_n \}$.
Let $W^+ := \{ w \in W : \ell_S(w) \mbox{ is even} \}$ (cf. Definition~\ref{def_length}(ii)), and put $x_i := s_0s_i$ for $1 \leq i \leq n$.
Then $W^+$ is the free group of rank $n$ and $\{ x_i : 1 \leq i \leq n \}$ is a basis of $W^+$ (cf. e.g. \cite[Proposition~2.1.1]{brenti_alternating}).
Furthermore, $\overline{W^+} := W^+/[W^+,W^+]$ is the free abelian group of rank $n$ with basis $\{ x_i/[W^+,W^+] : 1 \leq i \leq n \}$.
Also, by the considerations preceding the current lemma, we have homorphisms of monoids:
$$\pi: End(W^+) \rightarrow End(\overline{W^+})  \; \text{ and } \; det: End(\overline{W^+}) \rightarrow (\mathbb{Z}, \cdot).$$

\smallskip
\noindent
Notice now that for each $\alpha \in Sim(W,S)$ we have that $\alpha(W^+) \subseteq W^+$ and therefore we obtain a homomorphism
of monoids:
$$\varphi: Sim(W,S) \rightarrow End(W^+): \alpha \mapsto \alpha^+,$$ where $\alpha^+$ is the
restriction of $\alpha$ to $W^+$.

\smallskip
\noindent
For an odd prime $p$ let $w_p = s_ns_0s_ns_0 \cdots s_0s_n$ be the alternating product of length $2p-1$ and let
$\alpha_p \in Sim(W,S)$ be defined by $\alpha_p(s_i) := s_i$ for $i < n$ and $\alpha_p(s_n) := w_p$.
Then $det \circ \pi \circ \varphi(\alpha_p) = p$. Hence, by the considerations preceding the current lemma, we have that $M:= \varphi(Sim(W,S))$
is not finitely generated and therefore $Sim(W,S)$ is not finitely generated. This concludes the proof of the lemma.
\end{proof}

	\begin{remark}\label{the_final_hammer} Let $(W, S)$ be a right-angled Coxeter group of finite rank. We conjecture that $Sim(W,S)$ is {\em not} finitely generated whenever there are $s, t \in S$ satisfying the following assumptions:
	\begin{enumerate}[(1)]
	\item $o(st) = \infty$;
	\item $N(t) \subseteq N(s)$, where $N$ is as in Notation~\ref{star} for the Coxeter graph of $(W, S)$.
	\end{enumerate}
In support of this conjecture we observe that these two assumptions are exactly what we need in order for the maps $\alpha_p$ defined in the proof of Lemma~\ref{the_hammer} to be endomorphisms of $W$ (and so to be elements of $Sim(W,S)$). Explicitly, letting $S = \{ s_0, ..., s_n \}$, $s_0 = s$ and $s_n = t$, if $p$ is an odd prime and we let $w_p = s_ns_0s_ns_0 \cdots s_0s_n$ be the alternating product of length $2p-1$ and we let
$\alpha_p$ be defined by $\alpha_p(s_i) = s_i$ for $i \leq n$ and $\alpha_p(s_n) = w_p$. \mbox{Then, by Lemma~\ref{centralizer}, $\alpha_p \in Sim(W,S)$.}
\end{remark}

	\begin{proof}[Proof of Theorem~\ref{th_prime_models}] This is immediate by Lemmas~\ref{theorem_prime_model}, \ref{indecomposable_lemma} and \ref{the_hammer}.
\end{proof}

\section{Traces of Homogeneity in RACGs}

	In this section we prove Theorem~\ref{type_def_th}.

\medskip

	The material contained in this section is connected to the area of group theory which studies {\em test elements} (resp. test elements for monomorphisms), i.e. those $g \in G$ such that for every $\alpha \in End(G)$ (resp. for every monomorphism $\alpha \in End(G)$) we have that $\alpha(g) = g$ implies that $\alpha \in Aut(G)$. On this see e.g. \cite{turner1, turner2}.


%

	\begin{definition} Let $(W, S)$ be a right-angled Coxeter system of finite rank and let $S = \{s_1, ..., s_n \}$. We say that $h \in W$ is a Coxeter element of $(W, S)$ if there exists $\sigma \in Sym(\{1, ..., n\})$ such that $h = s_{\sigma(1)} \cdots s_{\sigma(n)}$. We say that $h$ is a Coxeter element of $W$ if it is a Coxeter element of $(W, S)$ for some Coxeter basis $S$ of $W$.
\end{definition}

	\begin{definition}\label{pre-special} Let $(W, S)$ be a right-angled Coxeter system. We say that $\alpha \in End(W)$ is a pre-special $S$-endomorphism when:
	\begin{enumerate}[(1)]
	\item $\alpha = \gamma \circ \beta$;
	\item $\beta \in F(\Gamma_S)$ (cf. Definition~\ref{F(Gamma)});
	\item $\gamma(t) \in t^W$, for every $t \in T = \alpha(S)$;
	\item $o(\gamma(t) \gamma(r)) = o(tr)$, for all $t, r \in T$.
\end{enumerate}
We denote the set of pre-special $S$-endomorphisms as $End_\star(W, S)$.
\end{definition}

	\begin{remark} Notice that by Proposition~\ref{muller_fact} if $\alpha \in End_\star(W, S)$, then the map:
$$W \rightarrow \langle \alpha(s) : s \in S \rangle_W: \; s \mapsto \alpha(s)$$
is an isomorphism, and so $\alpha$ is a monomorphism. Notice further that if $F(\Gamma) = Aut(\Gamma)$, then Definition~\ref{pre-special} simplifies, since in this case $\beta \in Aut(\Gamma)$ and $S = \alpha(T)$.
\end{remark}
	
	\begin{fact}[{\cite[Lemma~5.2]{caprace}}]\label{caprace_fact} Let $(W, S)$ be a reflection independent right-angled Coxeter system of finite rank and let $h$ be a Coxeter element of $(W, S)$. Then, for every $\alpha \in End_\star(W, S)$, the fact that $\alpha (h) = h$ implies that $\alpha \in Aut(W)$.
\end{fact}

	
	\begin{lemma}\label{lemma_for_homogeneity} Let $(W, S)$ be a right-angled Coxeter system of finite rank, $\bar{a} \in W^m$ and $H = H_{\bar{a}} = \langle \bar{a} \rangle_W$. Suppose that there exists $h \in H$ such that for every $\alpha \in End_\star(W, S)$ we have that $\alpha (h) = h$ implies that $\alpha \in Aut(W)$. Then $\bar{a}$ is type-determined, i.e. if $tp(\bar{a}/\emptyset) = tp(\bar{b}/\emptyset)$, then \mbox{there is $\alpha \in Aut(W)$ such that $\alpha(\bar{a}) = \bar{b}$.}
\end{lemma}

	\begin{proof} Let $S$ be a basis of $W$. First of all, for every $i \in [1, m]$, let:
	$$a_i = w_i(s_1, ..., s_n)$$
be normal forms in the alphabet $\{s_1, ..., s_n \}$. Let then $\bar{b} = (b_1, ..., b_m) \in W^m$ and suppose that $\bar{b}$ is not in the $Aut(W)$-orbit of $\bar{a}$. For the sake of contradiction, suppose also that $tp(\bar{a}) = tp(\bar{b})$.
Now, clearly we have:
$$ W \models \exists \bar{x} (\varphi_{\Gamma}(\bar{x}) \wedge \bigwedge_{i \in [1, m]} a_i = w_i (x_1, ..., x_n)). $$
Thus, since $tp(\bar{a}) = tp(\bar{b})$, we can find $\bar{t}' = (t'_1, ..., t'_n) \in W^m$ such that:
$$ W \models \varphi_{\Gamma}(\bar{t}') \wedge \bigwedge_{i \in [1, n]} b_i = w_i (t'_1, ..., t'_n). $$
Then there exists a basis $S' = \{s'_1, ..., s'_n\}$ of $W$ such that $T' = \{t'_1, ..., t'_n\}$ is a set of self-similar reflections of $(W, S')$. Hence, since $\bar{b}$ is not in the $Aut(W)$-orbit of $\bar{a}$, the monomorphism $\beta: s'_i \mapsto t'_i$ must be non surjective, and so we can find a formula $\chi(\bar{x}, \bar{y})$ as in the proof of Proposition~\ref{type_def_basis} such that $W \models \chi(\bar{t}', \bar{s}')$, and thus witnessing the non-surjectivity of $\beta$. But then, since $tp(\bar{a}) = tp(\bar{b})$, we have:
\begin{equation}\label{crucial_eq} W \models \exists \bar{x} \bar{y} (\varphi_{\Gamma}(\bar{x}) \wedge \bigwedge_{i \in [1, m]} a_i = w_i (x_1, ..., x_n) \wedge \varphi_{\Gamma}(\bar{y}) \wedge \chi(\bar{x}, \bar{y})). \tag{$\star$}
\end{equation}
Hence, if $\bar{t}$ is a witness for the quantifier $\exists \bar{x}$ in the formula in (\ref{crucial_eq}) and we let $\alpha: s_i \mapsto t_i$, for $i \in [1, n]$, then we have:
	\begin{enumerate}[(1)]
	\item the map $\alpha \in End_{\star}(W, S)$;
	\item $\alpha$ is not surjective and so $\alpha \notin Aut(W)$;
	\item $W \models a_i = w_i (t_1, ..., t_n)$.
	\end{enumerate}
Now from the above we infer:
$$\alpha(a_i) = \alpha(w_i(s_1, ..., s_n)) = w_i(\alpha(s_1), ..., \alpha(s_n)) = w_i(t_1, ..., t_n) = a_i.$$
Thus, $\alpha \restriction H = id_H$, a contradiction.
\end{proof}

	\begin{proof}[Proof of Theorem~\ref{type_def_th}] Immediate from Fact~\ref{caprace_fact} and Lemma~\ref{lemma_for_homogeneity}.
\end{proof}

\section{Prime Models and $\equiv$ in $2$-Spherical Coxeter Groups}

	In this section we prove Theorem~\ref{th_prime_2spherical}.

	\begin{lemma}\label{reflection_indep_2spherical} Let $(W, S)$ be an irreducible, $2$-spherical Coxeter system of finite rank. Then the set of reflections of $W$ is definable without parameters.
\end{lemma}

	\begin{proof} Let $G_1, ..., G_n$ be a list of the maximal special $S$-parabolic\footnote{Notice that in the proof of Lemma~\ref{reflection_indep_2spherical}, by the ``furthermore part'' of Fact~\ref{fact2_finite_cont} (i.e. by the strong rigidity of $W$), it does not matter the choice of Coxeter basis $S$ of $W$.} subgroups of $W$ (cf. Definition~\ref{def_parabolic}) and, for $\ell \in [1, n]$, let $|G_\ell| = k_\ell$. Then, by Fact~\ref{fact_finite_cont}, for every $w \in W$ of finite order, the finite continuation $FC(w)$ (cf. Definition~\ref{def_finite_cont}) of $w$ can be defined by the formula $\varphi(y, w)$ ($y$ is a free variable and $w$ is a parameter):
	\begin{enumerate}[(A)]
	\item for every $\ell \in [1, n]$ if there are $x_1, ..., x_{k_\ell} \in W$ such that $\{x_1, ..., x_{k_\ell} \}$ determines a subgroup $G'_\ell$ of $W$ isomorphic to $G_\ell$, and $G'_\ell $ contains $w$, then $y$ is in $G'_\ell$.
\end{enumerate}
Then, by Fact~\ref{fact2_finite_cont}, we can define $S^W$ to be the set of $x$ in $W$ such that $x$ is an involution, $x \neq e$ and $FC(x) = \{e, x \}$, and this is clearly a first-order condition.
\end{proof}

	\begin{proof}[Proof of Theorem~\ref{th_prime_2spherical}] Let $(W, S)$ be an irreducible, $2$-spherical Coxeter system of finite rank. Then, by Lemma~\ref{reflection_indep_2spherical}, the set $S^W$ is first-order definable without parameters. Suppose in addition that $(W, S)$ is even and not affine, and let $S = \{ s_1, ..., s_n\}$. Let then $\delta(x_1, ..., x_n)$ be the first-order formula expressing that:
	\begin{enumerate}[(i)]
	\item for every $i \in [1, n]$, $x_i$ is in $S^W$;
	\item for every $i \neq j \in [1, n]$, $o(x_ix_j) = o(s_is_j)$ (recall that $(W, S)$ is $2$- spherical).
\end{enumerate}
Let now $\bar{a} \in W^n$ be such that $W \models \varphi(\bar{a})$. Then the map $\alpha: s_i \mapsto a_i$ extends to an $S$-self-similarity of $W$, and so, by Proposition~\ref{muller_fact}, we have that $\alpha: W \rightarrow \langle \alpha(S) \rangle_W$ is an isomorphism. Now, by \cite{caprace_2spherical} we have that $W$ is co-Hopfian, and so it must be the case that $\alpha$ is actually an automorphism. Hence, $\delta(x_1, ..., x_n)$ defines the $Aut(W)$-orbit of any Coxeter generating set, i.e. $W$ is a prime model of its theory.
\end{proof}

	\begin{proof}[Proof of Corollary~\ref{el_equivalence_corollary}] Let $W_{\Gamma}$ and $W_{\Theta}$ be as in the assumptions of the corollary. Suppose now that $Th(W_\Gamma) = Th(W_\Theta)$, then $W_{\Gamma} \models Th(W_\Theta)$, and so, since $W_\Theta$ is a prime model of $Th(W_\Theta)$, we can find an elementary embedding $f: W_{\Theta} \rightarrow W_{\Gamma}$. Without loss of generality we may assume that $f$ is actually an inclusion map, so that $W_{\Theta} \preccurlyeq W_{\Gamma}$. Let now $\delta(\bar{x}) = \delta_{\Gamma}(\bar{x})$ be the formula from the proof of Theorem~\ref{th_prime_2spherical} for the Coxeter system $(W_{\Gamma}, S)$, where $S$ is any\footnote{Notice that also in the proof of Corollary~\ref{el_equivalence_corollary}, by the ``furthermore part'' of Fact~\ref{fact2_finite_cont} (i.e. by the strong rigidity of $W$), it does not matter the choice of Coxeter basis $S$ of $W$.} basis  for $W_{\Gamma}$. Then:
\[ \begin{array}{rcl}
 W_{\Gamma} \models \exists \bar{x} \delta_{\Gamma}(\bar{x}) 				& \Rightarrow & W_{\Theta} \models \exists \bar{x} \delta_{\Gamma}(\bar{x}) \\
	& \Rightarrow & \exists \bar{a} \in W^{|S|}_{\Theta} \text{ such that }
	W_{\Theta} \models \delta_{\Gamma}(\bar{a})\\
	& \Rightarrow &  W_{\Gamma} \models \delta_{\Gamma}(\bar{a})\\
	& \Rightarrow &  \bar{s} \mapsto \bar{a} \in Aut(W_{\Gamma}).\\
\end{array}	\]
Thus, we have:
$$\langle \bar{a} \rangle_{W_{\theta}} \leq W_{\Theta} \leq W_{\Gamma} \;\; \text{ and } \;\; \langle \bar{a} \rangle_{W_{\theta}} = W_{\Gamma}.$$
Hence, we must in fact have that $W_{\Theta} = W_{\Gamma}$ and so, by the ``furthermore part'' of Fact~\ref{fact2_finite_cont} (i.e. by the strong rigidity of $W$), we can conclude that $\Theta \cong \Gamma$.
	\end{proof}

\section{A Model-Theoretic Interpretation of Reflection Length}

	In this section we develop the model theoretic applications of reflection length announced in the introduction and in particular prove Theorem~\ref{first_theorem} and Corollaries~\ref{affine_corollary}~and~\ref{artin_theorem}. We invite the reader to recall Definition~\ref{reflection_length} and Facts~\ref{reflection_length_fact},~ \ref{abelianization_fact}.
	
	\begin{proof}[Proof of Theorem~\ref{first_theorem}]
	\underline{Item (\ref{item1})}. The fact that $N$ is characteristic in $G$ is clear, since automorphisms map involutions to involutions. We are then left to show that $N = \langle S^G \rangle_G$. To see this notice that every element $g \in W$ of order $2$ is a conjugate of an element in a spherical special parabolic subgroup of $W$ (see e.g. \cite{richardson}), i.e. there exists $S' \subseteq S$ such that $\langle S' \rangle_W = H$ is finite and $k \in H$ such that $g \in k^W$. Since there are only finitely many possibilities for such a $k$ (given that $S$ is assumed to be finite), we have that in $G$ it is true that if $g \in G$ is of order $2$, then $g = (s_1 \cdots s_n)^h$, for some $s_1, ..., s_n \in S$ and $h \in G$, and thus we have:
$$g = (s_1 \cdots s_n)^h = hs_1h^{-1}h s_2 h^{-1} \cdots hs_nh^{-1} \in \langle S^G \rangle_G.$$

\smallskip
\noindent \underline{Item (\ref{item3})}. This is by Fact~\ref{reflection_length_fact}(2) and the fact that the property ``$x$ has reflection length $\ell_T(x)$ at least $n$'' is first-order expressible over $S$.

\smallskip
\noindent \underline{Item (\ref{item4})}. This is by Fact~\ref{reflection_length_fact}(2) and the fact that the property ``$x$ has reflection length $\ell_T(x)$ at most $n$'' is first-order expressible over $S$.
\end{proof}

\begin{proposition}\label{second_theorem} Let $(W, S)$ be a Coxeter system of finite rank, $Q$ a group, and let $\eta: W \rightarrow Q \in Hom(W, Q)$ be such that $\eta(s^w) = \eta(s)$, for all $s \in S$ and $w \in W$. Then there exist $\psi_i(x, S)$, $i < \omega$, such that for every elementary extension $G$ of $W$:
	\begin{enumerate}[(1)]
	\item $\eta$ extends to an homomorphism $\hat{\eta}: N_G \rightarrow Q$ (cf. Theorem~\ref{first_theorem});
	\item $\bigvee_{i < \omega} \psi_i(G, S) = \ker(\hat{\eta})$;
	\item if $W$ is affine, then there exists $n < \omega$ such that $\bigvee_{i < n} \psi_i(G, S) = \ker(\hat{\eta})$.
\end{enumerate}
\end{proposition}

	\begin{proof} \underline{Item (1)}. Define:
	$$\hat{\eta}: N_G \rightarrow Q: s_1^{g_1} \cdots s_n^{g_n} \mapsto \eta(s_1 \cdots s_n).$$
Clearly, we have:
\[ \begin{array}{rcl}
	\hat{\eta}(s_1^{g_1} \cdots s_n^{g_n} t_1^{h_1} \cdots t_m^{h_m}) & = & \eta(s_1 \cdots s_n t_1 \cdots t_m) \\
	& = & \eta(s_1 \cdots s_n) \eta(t_1 \cdots t_m)\\
	& = &  \hat{\eta}(s_1^{g_1} \cdots s_n^{g_n}) \hat{\eta}(t_1^{h_1} \cdots t_m^{h_m}).\\
\end{array}	\]
Thus, we are left to show that $\hat{\eta}$ is well-defined. Suppose then that:
$$g = s_1^{g_1} \cdots s_n^{g_n} \;\; \text{ and } \;\; g = t_1^{h_1} \cdots t_m^{h_m}.$$
Then:
$$s_1^{g_1} \cdots s_n^{g_n} t_m^{h_m} \cdots t_1^{h_1} = e.$$
But then, since $G$ is an elementary extension of $W$, by Fact~\ref{abelianization_fact}, we have:
$$s_1 \cdots s_n t_m \cdots t_1 = e,$$
in fact the abelianization map from Fact~\ref{abelianization_fact} is a witness that in $W$ we have that if a product of reflections $s_1^{w_1} \cdots s_n^{w_n}$ equals $e$, then $s_1 \cdots s_n = e$.
Hence, we have:
$$\eta(s_1 \cdots s_n t_m \cdots t_1) = \eta(s_1 \cdots s_n) \eta(t_m \cdots t_1) = e.$$
Thus,
$$\eta(s_1 \cdots s_n) = \eta(t_1 \cdots t_m),$$
and so:
$$\hat{\eta}(s_1^{g_1} \cdots s_n^{g_n}) = \hat{\eta}(t_1^{h_1} \cdots t_m^{h_m}).$$
\underline{Item (2)}. By the definition of $\hat{\eta}$ from the proof of Item (1), it is clear that $ker(\hat{\eta})$ is defined by the following infinite disjunction of $S$-formulas:
$$ \bigvee \{ \exists y_1 \cdots \exists y_n(x = s_1^{y_1} \cdots s_n^{y_n}) : n < \omega, \; s_i \in S, \; \eta(s_1 \cdots s_n) = e \}.$$
\underline{Item (3)}. This follows from Item (2) and the boundedness of reflection length in affine Coxeter groups of finite rank from Fact~\ref{reflection_length_fact}(2).
\end{proof}

	\begin{proof}[Proof of Corollary~\ref{affine_corollary}] By Proposition~\ref{second_theorem} it suffices to show that the homomorphism $\varepsilon: W \rightarrow \{ +1, -1 \}$
defined as $s \mapsto -1$, for all $s \in S$, is such that $\varepsilon(s^w) = \varepsilon(s)$, for all $s \in S$ and $w \in W$. But by e.g.  \cite[Proposition~1.4.2]{brenti} we have that:
$$\varepsilon(s^w) = (-1)^{\ell_S(w)} 1 (-1)^{\ell_S(w^{-1})} = (-1)^{\ell_S(w)} 1 (-1)^{\ell_S(w)} = 1 = \varepsilon(s).$$
\end{proof}



%

	\begin{proof}[Proof of Corollary~\ref{artin_theorem}] We recall the construction from \cite{davis} and observe that it satisfies the assumption of our Proposition~\ref{second_theorem}. Let $\Gamma$ be a finite graph with domain $S = \{ s_i : i \in I \}$, we define a graph $\Gamma^+$ as follows. The domain of $\Gamma^+$ is $\{ s_i : i \in I \} \cup \{ r_i : i \in I \}$, where the sets $\{ s_i : i \in I \}$ and $\{ r_i : i \in I \}$ are disjoint. Concerning the adjacency relation of $\Gamma^+$ we have:
	\begin{enumerate}[(i)]
	\item $\{ s_i : i \in I \}$ spans a copy of $\Gamma$;
	\item for every $i \neq j \in I$ we have that $r_i$ is adjacent to $r_j$;
	\item for $i, j \in I$ we have that $s_i$ is adjacent to $r_j$ if and only if $i \neq j$.
\end{enumerate}	
Let now $(\mathbb{Z}/2)^I$ denote the direct sum of $I$ copies of a cyclic group of order $2$ and let $\bar{r}_i$ be the standard generators. Now, define $\theta: W(\Gamma^+) \rightarrow (\mathbb{Z}/2)^I$ by letting:
	$$\theta(s_i) = \theta(r_i) = \bar{r}_i,$$
for all $i \in I$ (cf. Fact~\ref{abelianization_fact}). Then, letting $A$ be the right-angled Artin group $A(\Gamma)$ on generators $\{ g_i : i \in I \}$, we have \mbox{that the map $\beta: A \rightarrow ker(\theta) \subseteq W(\Gamma^+)$ defined by:}
	$$\theta(g_i) = r_is_i,$$
for all $i \in I$, is an isomorphism. Thus, it suffices to show that the map $\theta$ satisfies the assumptions of Proposition~\ref{second_theorem}, but this is clear, since $(\mathbb{Z}/2)^I$ is abelian.

\end{proof}

\end{document}